\documentclass[11pt]{amsart}
\usepackage{amsmath,amscd,amssymb,amsfonts,amsthm,hyperref}

\usepackage{xcolor}  %the next few lines set the colors of the hyperlinks, and remove the color boxes
\hypersetup{   colorlinks,   linkcolor={red!50!black},    citecolor={blue!70!black},   urlcolor={blue!80!black}}

\usepackage[cmtip,matrix,arrow]{xy}
\newtheorem{theorem}{Theorem}[section]
\newtheorem{corollary}[theorem]{Corollary}
\newtheorem{proposition}[theorem]{Proposition}

\newtheorem{lemma}[theorem]{Lemma}
\theoremstyle{definition}    
\newtheorem{definition}[theorem]{Definition}
\theoremstyle{remark}

\newtheorem{remark}[theorem]{Remark}

\newtheorem{example}[theorem]{Example}
\newtheorem{examples}[theorem]{Examples}
\renewcommand{\AA}{\mathbb{A}}

\newcommand{\K}{\mathbb{K}}
\newcommand{\W}{\mathcal{W}}

\renewcommand{\L}{\mathcal{L}}

\newcommand{\ca}{\mathcal}

\newcommand{\X}{\mathcal{X}}
\newcommand{\F}{\mathcal{F}}

\newcommand{\R}{\mathbb{R}}

\newcommand{\Z}{\mathbb{Z}}

\newcommand\lie[1]{\mathfrak{#1}}
\renewcommand{\k}{\lie{k}}

\newcommand{\g}{\lie{g}}

\newcommand{\on}{\operatorname}

\newcommand{\Aut}{ \on{Aut} }

\newcommand{\Hom}{ \on{Hom}}

\renewcommand{\ker}{ \on{ker}}

\newcommand\qu{/\kern-.7ex/} % Categorical quotients
\newcommand{\hra}{\hookrightarrow}

\renewcommand{\d}{{\mbox{d}}}

\newcommand{\f}{\frac}

\newcommand{\p}{\partial}

\newcommand\hh{{\f{1}{2}}}

\newcommand{\ti}{\tilde}
\newcommand{\eeq}{\end{eqnarray*}}
\newcommand{\beq}{\begin{eqnarray*}}

\renewcommand{\H}{\ca{H}}
\renewcommand{\K}{\ca{K}}

\newcommand{\D}{\ca{D}}

\newcommand{\lin}{{\on{lin}}}

\newcommand{\mf}{\mathfrak}
\newcommand{\rra}{\rightrightarrows}
\renewcommand{\subset}{\subseteq}
\renewcommand{\supset}{\supseteq}
\newcommand{\I}{\ca{I}}

\newcommand{\wt}{\widetilde}

\newcommand{\su}{\mathsf{u}}
\newcommand{\sj}{\mathsf{j}}
\newcommand{\si}{\mathsf{i}}

\newcommand{\alg}{{\on{alg}}}
\newcommand{\gr}{\on{gr}}

\begin{document}

%\sloppy
\title{Singular Lie filtrations and weightings} 
\author{Yiannis Loizides}
\author{Eckhard Meinrenken}

\begin{abstract}  We study weightings (a.k.a. quasi-homogeneous structures) arising from manifolds with singular Lie filtrations. This generalizes constructions of Choi-Ponge, Van Erp-Yuncken, and Haj-Higson for (regular) Lie filtrations.   		
\end{abstract}
\maketitle

%\begin{quote}	{\it \small Dedicated to Victor Guillemin, on the occasion of his 85th birthday.}\end{quote}
\vskip1cm
\bigskip

%\tableofcontents

\section{Introduction}
A Lie filtration on a manifold $M$ is a $\Z$-filtration of the tangent bundle $TM$ by subbundles 
\begin{equation}\label{eq:lifiltration1} TM=H_{-r}\supset \cdots \supset H_{-1}\supset 0,\end{equation}
in such a way that that the induced filtration on the Lie algebra of vector fields $\mf{X}(M)$ is compatible with brackets. The concept of such \emph{filtered manifolds} (called \emph{Carnot manifolds} in \cite{choi:tan}) was introduced by T.~Morimoto \cite{mor:geo} as a generalization of the differential systems of Tanaka \cite{tan:dif}; independently, the concept was considered by Melin \cite{mel:lie}. Filtered manifolds have been much studied in recent years as a framework for certain types of hypo-elliptic operators. Roughly, these operators are elliptic in a \emph{weighted} sense, where weights are assigned according to the Lie filtration. Unlike the usual elliptic setting, the principal symbols of these operators do not live on the usual tangent bundle, but on the
 \emph{osculating groupoid} \cite{erp:tan} (also called \emph{tangent cone} \cite{choi:tan}). The latter is a 
 family of nilpotent Lie groups $P\to M$, obtained by integrating the family of Lie algebras
  \[ \mathfrak{p}=\on{gr}(TM)=\bigoplus_i\ H_{-i}/H_{-i+1}.\]  

An index theorem for this context was recently obtained by Mohsen \cite{moh:ind}, generalizing Van Erp's index theorem for contact manifolds \cite{erp:ati1,erp:ati2}. The proofs are based on Connes' \emph{tangent groupoid} strategy \cite{con:non}. The relevant deformation groupoid  was constructed 
by Choi-Ponge \cite{choi:tan}, Haj-Higson \cite{hig:eul}, Van Erp-Yuncken \cite{erp:tan}, and Mohsen \cite{moh:def}, using different techniques. Later work of Van Erp-Yuncken \cite{erp:gro}, building on ideas of Debord-Skandalis \cite{deb:lie}, used this viewpoint to develop a pseudo-differential calculus for this context. 

Haj-Higson \cite{hig:eul} extended the construction of  osculating groupoid to define a normal cone, and corresponding deformation space, for \emph{filtered submanifolds} $N\subset M$ of filtered manifolds. This was obtained by using a filtration on the algebra of functions on $M$ by a weighted order of vanishing, as determined by the Lie filtration. The normal cone is described algebraically as the character spectrum of the  associated graded algebra, and the 
deformation space  is the character spectrum of the Rees algebra associated to the filtration. Haj-Higson proved that 
the normal cone is a quotient $P|_N/R$, where $P$ and $R$ are the osculating groupoids of $M$ and $N$, respectively. 

In \cite{loi:wei}, we found that the construction of the normal cone and the associated deformation space only requires a much weaker notion of \emph{weighting along $N$}. This concept was introduced by Melrose under the name of \emph{quasi-homogeneous structure}, in his lecture notes on manifolds with corners \cite{mel:cor}. One possible description of weightings is in terms of a filtration on the algebra of functions, and allows for a definition of weighted normal bundle 
in terms of the associated graded algebra
\begin{equation}\label{eq:weightednormalbundle} \nu_\W(M,N)=\Hom_\alg(\on{gr}(C^\infty(M)),\R)\end{equation}
just as in \cite{hig:eul}. 
Equivalently,  $r$-th order weightings admit a description in terms of certain subbundles $Q\to N$ of the $r$-th tangent bundle $T_rM\to M$. From this perspective,  the weighted normal bundle is realized as a quotient 
\begin{equation}\label{eq:weightedn} \nu_\W(M,N)=Q/\sim\end{equation}
generalizing the familiar description of the usual normal bundle as $\nu(M,N)=TM|_N/\sim$. 

Weighted normal bundles (and the associated deformation spaces) are tailor-made for discussions of weighted normal forms of geometric structures, such as those arising in singularity theory 
(see e.g. Golubitsky-Guillemin \cite{gol:sta}). Weightings also allow for a coordinate-free definition of \emph{weighted blow-ups}, which play an important role in algebraic and symplectic geometry (see e.g., the result of Guillemin-Sternberg \cite{gu:bir} on birational equivalence).  Given a collection $N_1,\ldots,N_j$ of submanifolds of $M$ that intersect cleanly (in the sense that there are local charts in which the $N_i$ are given by coordinate subspaces), one obtains a multi-weighting given by the order of vanishing on the $N_i$, and an associated \emph{total weighting}. This type of situation appears, for example, in orbit  type decompositions of $G$-manifolds and their resolutions (see Albin-Melrose \cite{alb:res}). 

In the present article, we shall study weightings in the context of  \emph{singular Lie filtrations}, given by a filtration of the sheaf of vector fields,
\begin{equation}\label{eq:singularliefiltration} \mf{X}_M=\H_{-r}\supset\cdots \supset \H_{-1}\supset 0\end{equation}
%Note that we are allowing for a non-zero $\H_0$; this component defines a \emph{singular foliation} in the sense of Androulidakis-Skandalis.  
by locally finitely generated $C^\infty_M$-submodules, with $[\H_{-i},\H_{-j}]\subset \H_{-i-j}$. 
Singular Lie filtrations appear in a large variety of contexts. They arise as Carnot structures in sub-Riemannian geometry \cite{agr:com,bel:tan}, and play a central role  in recent work of Androulidakis, Mohsen, Van Erp and Yuncken \cite{and:inprep} on 
 hypo-elliptic operators, encompassing a broader class than those associated with regular Lie filtrations.

In Section \ref{subsec:sing}, we introduce the notion of a  \emph{$\H_\bullet$-clean}  submanifold $N\subset M$, 
which essentially means that sections of the normal bundle $\nu(M,N)$ given as images of $\H_{-i}$ are in fact 
subbundles. One of our main results is the following (see Theorem \ref{th:mainresult}):
\bigskip

\noindent{\bf Theorem.} Let $M$ be a manifold with a singular Lie filtration $\H_\bullet$, and  let 
 $N\subset M$ be an $\H_\bullet$-clean submanifold. Then $M$ acquires a canonical weighting along $N$, in such a way that the vector fields in $\H_{-i}$ have filtration degree $-i$.
\bigskip 

The proof of this result amounts to a construction of local coordinates adapted to the singular Lie filtration. To describe the subbundle $Q\subset T_rM$ corresponding to the weighting, we  observe that a singular Lie filtration of $M$ determines a singular foliation of $T_rM$. We then show that $Q$ is obtained as the `flow-out' of $T_rN\subset T_rM$ under this singular foliation; here the cleanness assumption guarantees that this flow-out is smooth. 

This gives a first description of the weighted normal bundle as a quotient \eqref{eq:weightedn}. We will also give a more direct description, by associating to each $m\in M$ a nilpotent Lie algebra $\mathfrak{p}_m$ given as a pullback of the sheaf of Lie algebras $\bigoplus\H_{-i}/\H_{-i+1}$. If $m\in N$, we define $\mathfrak{r}_m$ similarly, by using vector fields in $\H_{-i}$ that are tangent to $N$ . 
Letting $P_m\supset R_m$ be the corresponding groups, we show that the fibers of the weighted normal bundle are the homogeneous spaces $P_m/R_m$. For the case of a regular Lie filtration, this recovers the result of Haj-Higson mentioned above.

\section{Weightings}
The basic idea of a weighting is to have a notion of `weighted order of vanishing' along submanifolds $N\subset M$. 
This appeared in work of Melrose \cite{mel:cor} under the name of \emph{quasi-homogeneous structure}. 
Some foundational aspects of weightings, such as the concepts of \emph{weighted normal bundles} and \emph{weighted deformation spaces}, were developed in \cite{loi:wei}. Let us briefly summarize some of this material, starting with the definitions.  We adopt a sheaf-theoretic language, so that all constructions carry over to the holomorphic or analytic categories with straightforward changes.   

\subsection{Definitions}
Let  $w_1,\ldots,w_n\in  \Z_{\ge 0}$ be a given sequence of \emph{weights}. An upper bound for the weight sequence will be call its \emph{order}.

For open subsets $U\subset \R^n$, consider the filtration of the algebra $C^\infty(U)$ of smooth functions, 
where  $C^\infty(U)_{(i)}$ is the ideal of functions generated by monomials $x^s=x_1^{s_1}\cdots x_n^{s_n}$ with  $s\cdot w=s_1w_1+\ldots+s_nw_n\ge i$. A \emph{weighted atlas} on an $n$-dimensional manifold $M$ is given by coordinate charts such that the transition functions between two charts are filtration preserving; a 
maximal weighted atlas is a \emph{weighting on $M$}. The local coordinates from a weighted atlas are  called \emph{weighted coordinates}. The weighting determines a closed submanifold 
\[ N\subset M,\] 
given in local weighted coordinates as the vanishing set of the coordinates $x_a$ such that $w_a>0$. In particular, the weighted coordinates are submanifold coordinates for $N$. The weighting gives  a  filtration of the function sheaf by ideals
\[ C^\infty_M=C^\infty_{M,(0)}\supset C^\infty_{M,(1)}\supset \cdots\]
where  $C^\infty(U)_{(i)}$ has the description above whenever  $U\subset M$ is the domain of a weighted coordinate chart. 
The definition implies that the filtration on functions is multiplicative:
\[ C^\infty_{M,(i)}C^\infty_{M,(j)}\subset C^\infty_{M,(i+j)}\] 
%for all $i,j\ge 0$), 
and that  $C^\infty_{M,(1)}$ is the vanishing 
ideal sheaf $\I$ of  $N$. We think of the filtration as giving the weighted order of vanishing along $N$, and we speak of a \emph{weighting along $N$}. 
%The \emph{trivial weighting along $N$} corresponds to the case $r=1$: all the weights are zero or one. 
Furthermore, we obtain a filtration of the normal bundle
\begin{equation}\label{eq:filtration} \nu(M,N)=F_{-r}\supset \cdots \supset F_{-1}\supset F_0=0\end{equation}
where for all $i\ge 1$, 
\begin{equation}\label{eq:quotient1}
C^\infty_{M,(i+1)}/C^\infty_{M,(i+1)}\cap \I^2
=\Gamma_{\on{ann}(F_{-i})}
\end{equation}
where $\on{ann}(F_{-i}\subset \nu(M,N)^*$ is the annihilator bundle. That is, the differentials of functions of filtration degree $i+1$ vanish in the direction of $F_{-i}$. 
Note that  $\dim F_{-i}=\#\{w_a\colon w_a\le i\}$. 

Conversely, weightings of order $r$ along $N$ are characterized as multiplicative filtrations of $C^\infty_M$ with 
$C^\infty_{M,(1)}=\I$, with the property  \eqref{eq:quotient1} for a suitable filtration \eqref{eq:filtration} of the normal bundle, 
%the properties and $C^\infty_{M,(r+1)}\subset \I^2$, such that \eqref{eq:quotient1} are sheaves of sections of subbundles 
%of the conormal bundle, 
and with the additional property that for $i>1$, 
\begin{equation}\label{eq:extra}  C^\infty_{M,(i)}\cap \I^2=\sum_{j<i}    C^\infty_{M,(j)}\cdot C^\infty_{M,(i-j)}.\end{equation}
%The inclusion $\supset$ in (b) is automatic, by multiplicativity; the equality means that for any given $i>1$, the ideal $\I^2$ contains no `unexpected' elements of filtration degree $i$. 
The filtration on the sheaf of functions determines a filtration on the sheaves of differential forms, vector fields, and other tensor fields. In particular, 
\begin{equation}\label{eq:vf_filtration}
 \mf{X}_M=\mf{X}_{M,(-r)}\supset \mf{X}_{M,(-r+1)}\supset  \cdots\supset  \mf{X}_{M,(0)}\supset \cdots\end{equation}
where $\mf{X}_{M,(j)}$ is the sheaf of vector fields $X$ with the property that the Lie derivative on functions raises the filtration degree on functions by $j$. The sections of $\mf{X}_{M,(0)}$ are the infinitesimal automorphisms of the weighting. 
Letting $\mf{X}_M^N$ be the subsheaf of vector fields tangent to $N$, with its induced filtration, we have 
that 
\[ \mf{X}_{M,(j)}/\mf{X}^N_{M,(j)}=\Gamma_{F_j}.\]
%is the sheaf of sections of $F_j$. 
(For $j\ge 0$ we have $\mf{X}_{M,(j)}=\mf{X}^N_{M,(j)}$, since infinitesimal automorphisms of the weighting are in particular tangent to $N$.)  
\begin{remark}
One can generalize the definition to non-closed or immersed submanifolds $\si\colon N\to M$ by applying the definition above to pairs of open subsets $(U,V)\subset (M,N)$ where $\si(V)$ is a closed embedded submanifold of $U$, and requiring agreement between the local weightings on overlaps. Globally one obtains a filtration of the pullback sheaf $\si^{-1}C^\infty_M$.
\end{remark}

\subsection{The weighted normal bundle}
For any closed submanifold $N\subset M$, with vanishing ideal $\I\subset C^\infty_M$, the sheaf 
$\bigoplus_i \I^i/\I^{i+1}$ may be regarded as the sheaf of fiberwise polynomial functions on the normal bundle 
$\nu(M,N)=TM|_N/TN$. Such functions are the (local) sections of a graded algebra bundle 
$\mathsf{A}\to N$, and the normal bundle is obtained by taking its fiberwise spectrum, 
$\nu(M,N)|_m=\Hom_\alg(\mathsf{A}_m,\R)$. 
\subsubsection{Definition of weighted normal bundle}
This algebraic description of the normal bundle generalizes to the weighted case. Given a weighting of $M$ along $N$,  
the sheaf of associated graded algebras $\gr (C^\infty_M)$ is supported on $N$, and is the sheaf of sections 
of a graded algebra bundle $\mathsf{A}\to N$, with components $\mathsf{A}^i\to N$ of finite rank. In weighted 
coordinates over $U\subset M$, the space of sections of $\mathsf{A}^i|_{N\cap U}$ is spanned  by monomials $x^s$ such that $s\cdot w=i$, with  $s_b=0$ for $w_b=0$. We define 
 a  \emph{weighted normal bundle}
\[ \nu_\W(M,N)\to N\]
by taking the fiberwise spectrum:  $\nu_\W(M,N)|_m=\Hom_\alg(\mathsf{A}_m,\R)$. This is naturally a smooth fiber 
bundle of rank equal to the codimension of $N$ in $M$; the smooth structure on $\nu_\W(M,N)$ is uniquely determined by the property that for any given element of $\gr(C^\infty(U))$, the corresponding function on $\nu_\W(M,N)|_{N\cap U}$ is again smooth. 

\subsubsection{Graded bundles}\label{subsec:graded}
The weighted normal bundle  does not have a natural vector bundle structure, in general. It does, however, carry an action of the monoid of multiplicative scalars, 
\begin{equation} \label{eq:kappa}
 \kappa\colon \R\times \nu_W(M,N)\to \nu_\W(M,N),\ \ \kappa(t,\mathsf{x})=\kappa_t(\mathsf{x}).\end{equation}
This is induced fiberwise by the action of $t\in \R$ on 
$\mathsf{A}^i|_m$ as multiplication by $t^i$. Grabowski-Rotkiewicz \cite{gra:hig,gra:gra} refer to a smooth manifold $E$ with a monoid action 
$\kappa\colon \R\times E\to E$ as a  \emph{graded bundle}. As the name suggests, a graded bundle  is automatically a fiber bundle over $N=\kappa_0(M)$. Negatively graded vector bundles $V=\bigoplus_{i>0} V^{-i}\to N$ are special cases, 
where $\kappa_t$ is given on $V^{-i}$ as multiplication by $t^i$. (The reason for using a \emph{negative} grading is that we want the dual space, consisting of linear functions on $V$, to be positively graded.) Other basic examples include the higher tangent bundles $T_rM\to M$ to be discussed in Section \ref{subsec:trm}. 
%A function on a graded bundle $E\to N$ that is homogeneous of degree $i$ will be called a (fiberwise) homogeneous polynomial of degree $i$. Taking the sum over all $i$, this defines the sheaf of fiberwise polynomial functions.Thus, by definition, $\nu_\W(M,N)$ is the unique graded bundle having $\gr(C^\infty_M)$ as its sheaf of fiberwise polynomial functions. 
The \emph{linear approximation} of a graded bundle is defined by application of the (usual) normal bundle functor, 
\[ E_\lin=\nu(E,N).\] 
It has the structure of a graded \emph{vector} bundle over $N$, with the grading obtained by applying the normal bundle functor to $\kappa$.  As shown in \cite{gra:hig,gra:gra}, there always 
exists an isomorphism of graded bundles $E\cong E_\lin$, but such an isomorphism  (called a linearization) is not unique.
In the case of $E=\nu_\W(M,N)$, 
for a given weighting of $M$ along $N$, the linear approximation is \cite[Proposition 4.4]{loi:wei}
\[ \nu_\W(M,N)_\lin=\on{gr}(\nu(M,N)),\]
the associated graded bundle for the filtration \eqref{eq:filtration}.  

\subsubsection{Homogeneous approximations}
For $f\in C^\infty(U)_{(i)}$ let $f^{[i]}\in \gr^i(C^\infty(U))$ be its image. By definition of the weighted normal bundle, it may (and will) be regarded as a function  
$f^{[i]}\in C^\infty(\nu_W(M,N)|_{N\cap U})$. This function is  homogeneous of degree $i$ with respect to the $(\R,\cdot)$-action, and is called the \emph{homogeneous approximation} of $f$.  For $i=0$, 
$f^{[0]}$ is the pullback of the restriction of $f|_{N\cap U}$. More generally, 
any tensor field $\alpha$ of filtration degree $i$ on $U$ determines a tensor field $\alpha^{[i]}$, homogeneous of degree $i$, on $\nu_\W(M,N)$. In particular, this applies to vector fields and differential forms; note also that  
homogeneous approximation is compatible with the usual operations from Cartan's calculus. 

If $X$ is a vector field of strictly negative filtration degree $j<0$, then the vector field 
$X^{[j]}$ on $\nu_\W(M,N)$ is \emph{vertical} (tangent to the fibers of $\nu_\W(M,N)\to N$). To see this, it suffices to note that $X^{[j]}$ vanishes on pullbacks of functions, which in turn follows from  $X^{[j]}f^{[0]}=(Xf)^{[j]}=0$ for 
$j<0$. 

If $x_1,\ldots,x_n$ are local weighted coordinates on 
$U\subset M$ (thus $x_a\in C^\infty(U)$ has weight $w_a$), then the functions 
\[ x_1^{[w_1]},\ldots,\ x_n^{[w_n]}\] 
serve as local coordinates on $\nu_\W(M,N)|_U=\nu_\W(U,U\cap N)$; the homogeneous lifts of the corresponding coordinate vector fields are
\[ (\f{\p}{\p x_a})^{[-w_a]}=\f{\p}{\p x_a^{[w_a]}}.\] 

\begin{example}
	Consider the case of a symplectic manifold $(M,\omega)$ of dimension $2n$
	with an isotropic submanifold $N\subset M$ of dimension $k$. Let 
	$\I$ be the vanishing ideal of $N$. %Note that this is not closed under Poisson bracket (unless $N$ is Lagrangian). 
	Let $\ca{N}(\I)$ be the Poisson normalizer of $\I$, i.e., the sheaf of functions $f$ such that $\{f,\cdot\}$ preserves $\I$. 
	The intersection $\I\cap \ca{N}(\I)$ is an ideal in $C^\infty_M$, and we obtain a weighting of order $r=2$, with 
	\[ C^\infty_{M,(1)}=\I,\ \ C^\infty_{M,(2)}=\I\cap \ca{N}(\I).\]
	The resulting filtration of the normal bundle is 
	\[ \nu(M,N)\supset TN^\omega/TN\supset 0.\]
	By standard normal form theorems, there exists local coordinates $q_1,\ldots,q_n,p_1,\ldots,p_n$ near any given point of 
	$N$ such that the vanishing ideal $\I$ of $N$ is spanned by $q_{k+1},\ldots,q_n,p_1,\ldots,p_n$. In such coordinates, the ideal 
	$\I\cap \ca{N}(\I)$ is spanned by $\I^2$ together with $p_1,\ldots,p_k$. The coordinates $q_1,\ldots,q_k$ have weight $0$, the dual coordinates $p_1,\ldots,p_k$ have weight $2$, and the remaining coordinates $q_a,p_a$ for $a>k$ have weight $1$. Note that the 
	symplectic form has filtration degree $2$ for this weighting, as is evident from the coordinate description 
	$\omega=\sum \d q_i\wedge \d p_i$. Hence it has a homogeneous approximation $\omega^{[2]}\in \Omega^2(\nu_\W(M,N))$, which is again symplectic. See \cite{me:eul} for a discussion of Weinstein's isotropic embedding theorem from this perspective. 
	
	%For a description of the Lie algebra bundle $\k$, note that  the symplectic form gives isomorphisms $\omega^\flat\colon \mf{X}_{M,(-i)}\to \Omega^1_{M,(-i+2)}$. Hence, $ \Gamma_{\k^{-i}}=\Omega^1_{M,(-i+2)}/\Omega^1_{M,(-i+3)}$ for $i=1,2$. Since $\Omega^1_{M,(1)}$ is the sheaf of 1-forms whose pullback to $N$ is zero, this shows $\k^{-2}=T^*N$. On the other hand, \[  \Gamma_{\k^{-1}}=\Omega^1_{M,(1)}/\Omega^1_{M,(2)}.\]The bracket $[\cdot,\cdot]\colon \k^{-1}\times \k^{-1}\to \k^{-2}$ is induced by $[\alpha_1,\alpha_2]=\iota_N^*(\omega^\flat([X_{\alpha_1},X_{\alpha_2}]))$, $\omega(X_\alpha,\cdot)=\alpha$.

\end{example}

\subsubsection{Alternative description of weighted normal bundle}\label{subsec:homogeneous}
The weighting also determines a negatively graded Lie algebra bundle
\[ \k=\bigoplus_{i=1}^r \k^{-i}\] where $\k^{-i}$ has $\on{gr}(\mf{X}_M)^{-i}$
as its sheaf of sections, using the filtration \eqref{eq:vf_filtration}. From the filtered subsheaf  $\mf{X}_M^N\subset \mf{X}_M$ of vector fields that are tangent to $N$, we obtain a 
graded Lie subalgebra bundle
\[ \mf{l}=\bigoplus_{i=1}^r \mf{l}^{-i},\]
where  $\mf{l}^{-i}$ has $\on{gr}(\mf{X}_M^N)^{-i}$
as its sheaf of sections. The quotient bundle $\k/\mf{l}$ is identified with $\on{gr}(\nu(M,N))=\bigoplus F_{-i}/F_{-i+1}$, 
using the filtration \eqref{eq:filtration} of $\nu(M,N)$. 

\begin{remark}
In \cite{loi:wei} we defined $\mf{l}$ in terms of the subsheaf of vector fields 
that vanish along $N$, rather than those which are tangent to $N$. 
To see that the definitions are equivalent, note that $\mf{X}_{M,(0)}^N=\mf{X}_{M,(0)}$ since infinitesimal automorphisms of the weighting are tangent to $N$,  and that the restriction map $ \mf{X}_{M,(0)}\to \mf{X}_{N}$ is surjective. This implies that 
\[ \mf{X}_{M,(-i)}^N=\mf{X}_{M,(0)}^N+\I\mf{X}_M\cap \mf{X}_{M,(-i)},\ \ \ i\ge 0,\] and hence 
$\on{gr}(\mf{X}_M^N)^{-i}=\on{gr}(\I\mf{X}_M)^{-i}$
for all $i>0$.
\end{remark}
\begin{remark}
 In local weighted coordinates $x_a$ on $U\subset M$, the space of sections of $\k^{-i}|_{U\cap N}$ is spanned by vector fields $x^s\f{\p}{\p x_a}$ with $w_a>0$, ranging over multi-indices 
with $s_b=0$ for $w_b=0$ and with $w\cdot s-w_a=-i$. The subspace of sections of 
$\mf{l}^{-i}|_{U\cap N}$ is given by the additional condition that $s\neq 0$; hence the quotient $\k^{-i}/\mf{l}^{-i}$ is spanned by 
coordinate vector fields $\f{\p}{\p x_a}$ with $w_a=i$. Note in particular that $\mf{l}^{-r}=0$. 
\end{remark}

Since the monoid action of $(\R,\cdot)$ on $\k$ preserves brackets, it exponentiates to  a monoid  action on the nilpotent Lie group bundle $K\to N$ integrating $\k$; thus $K$ is an example of a  \emph{graded Lie group bundle}. 
Similarly $\mf{l}$ integrates to a graded Lie subgroup bundle $L\subset K$. One obtains the following 
description of the weighted normal bundle as a graded bundle of homogeneous spaces, 
\begin{equation}\label{eq:quot} \nu_\W(M,N)=K/L,\ \ \ \nu_\W(M,N)_\lin=\k/\mf{l}.\end{equation}
See \cite[Proposition 7.7]{loi:wei}. For the case of a trivial weighting ($r=1$), we directly have $\k=\nu(M,N)$ (with zero bracket) and $\mf{l}=0$, hence $K/L=K=\nu(M,N)$.

\section{Singular Lie filtrations}
In this section, we unify the concept of Lie filtrations with the concept of a singular foliation. These \emph{singular Lie filtrations}, to be discussed below,  appear in the work of Androulidakis, Mohsen,  Yuncken, and van Erp 
	on hypo-elliptic operators \cite{and:conv,and:inprep}. %As observed by these authors, singular Lie filtrations on $M$ may be re-interpreted as singular foliations on $M\times (0,\infty)$. 
 As we shall see, they  provide a rich source of examples of weightings. 

\subsection{Singular distributions}\label{subsec:sing}
A  \emph{distribution} on a manifold $M$ is a subbundle $D\subset TM$ of the tangent bundle. It is called Frobenius integrable if its space of sections is closed under Lie bracket. In the Stefan-Sussman theory of singular foliations, one considers more general families of subspaces $D_m\subset T_mM$ which are not necessarily of constant rank. Following work of Androulidakis-Skandalis \cite{and:hol}, it was found to be more useful to work with the sheaf $\mf{X}_M$ of vector fields, regarded as a sheaf of $C^\infty_M$-modules.   The formulation in \cite{and:hol} is in terms of vector fields of compact support; the equivalence with the sheaf-theoretic formulation is discussed in \cite{and:ste}.\smallskip

\begin{definition} Let $M$ be a manifold. 
	\begin{enumerate}
		\item 
			A \emph{singular distribution on $M$} is a sheaf of $C^\infty_M$-submodules $\D\subset \mf{X}_M$ that is \emph{locally finitely generated}. That is, every point in $M$ admits an open neighborhood $U$ such that $\D(U)$ is finitely generated as a 
		$C^\infty(U)$-module. 
		\item 
			A \emph{singular foliation on $M$} is a singular distribution $\D$ that is \emph{involutive}: $[\D,\D]\subseteq \D$. 
		\item 
		If $\D$ is the sheaf of sections of a subbundle $D\subset TM$, we speak of a \emph{regular distribution}. It is called a
		\emph{regular foliation}  if $\D$ is involutive. 	
	\end{enumerate}
\end{definition}
%A singular foliation is a locally finitely generated sheaf of submodules which is also a sheaf of Lie subalgebras: That is, $[\H,\H]\subset \H$. 
\smallskip

The sheaf formulation entails a gluing property: If a vector field $X\in\mf{X}(U)$ is such that every point of $U$ has an open neighborhood $U'\subset U$ with $X|_{U'}\in \D(U')$, then $X\in \D(U)$.  

\begin{remark}Given a singular foliation, one obtains a decomposition of 
$M$ into leaves, such that $\D$ spans the tangent spaces to the leaves. However, the submodule $\D$ contains more information, in general, than the decomposition into leaves. (For example \cite{and:hol}, the vector fields $x^2\f{\p}{\p x}$ and 
$x\f{\p}{\p x}$ span different submodules of $\mf{X}_\R$, but yield the same decomposition into leaves.) 
\end{remark}

%\subsection{Basic constructions}\label{subsec:basicproperties}
%
Here are some simple constructions with singular distributions. 
First, we note that if $M',M''$ are equipped with singular distributions $\D',\D''$, then the direct product $M'\times M''$ inherits a 
product distribution 
$\D'\times \D''\subset \mf{X}_{M'\times M''}$ 
defined by 
\[ (\D'\times \D'')(U'\times U'')=C^\infty(U'\times U'')\cdot (\D'(U')\oplus \D''(U'')).\]
If $\D',\D''$ are involutive then so is their product. 
Next, consider the restriction of singular distributions to   embedded submanifolds 
	$N\subset  M$. 
	\begin{definition}
		Let  $\D\subset \mf{X}_M$ be a singular distribution. 
		We say that $N\subset M$ is
		\begin{enumerate}
			\item   \emph{$\D$-transverse} if $ T_mN+ \D|_m=T_mM$ for all $m\in N$,
			\item   \emph{$\D$-invariant} if $\D|_m\subset T_mN$  for all $m\in N$,	
			\item 	 \emph{$\D$-clean}  if  $ \dim(T_mN+ \D|_m)$ is constant, as a function of $m\in N$. 
		\end{enumerate}		
	\end{definition}
Clearly, properties (a),(b) are special cases of property (c). 
\begin{remark}
	If $\D$ is a regular distribution, given as the sheaf of sections of  a subbundle $D\subset TM$, the cleanness condition is equivalent to $D|_N\cap TN$ being a subbundle of $TN$.  
\end{remark}

Suppose  $\D\subset \mf{X}_M$ is a singular distribution, and $N\subset M$ is 
$\D$-clean. Then we obtain a subbundle $\wt{F}\subset TM|_N$ with fibers
\[ \wt{F}_m=T_mN+\D|_m,\]
and a corresponding subbundle of the normal bundle 
\begin{equation}\label{eq:f} F=\wt{F}/TN\subset \nu(M,N)=TM|_N/TN.\end{equation}
The $\D$-transverse and $\D$-invariant cases are the special cases for which $F$ is the full normal bundle or the zero bundle, respectively. Let $\D^N=\D\cap \mf{X}_M^N$ the subsheaf tangent to $N$, and denote by 
\[ \si^!\D\subset \mf{X}_N\] 
its image under restriction $\mf{X}_M^N\to \mf{X}_N$. 
%by 
Explicitly, for $U\subset M$ open, 
\[ (\si^!\D)(U\cap N)=\{X|_{U\cap N}|\ X\in \D(U) \mbox{ is tangent to }U\cap N\}.\]
%this is the sheaf of  vector fields on $N$ that are $\si$-related to vector fields on $M$. 
	
	\begin{lemma}\label{lem:clean}
		Let $M$ be a manifold with a singular distribution $\D\subset \mf{X}_M$. 
		If $\si\colon N\to M$ is $\D$-clean, then $\si^!\D\subset \mf{X}_N$ is again
		a singular distribution on $N$. If $\D$ is a singular foliation on $M$, then $\si^!\D$ is a singular foliation on $N$. 
	\end{lemma}
	\begin{proof}
		The clean intersection condition  is equivalent to 
		\[ q=\dim(\D|_m)-\dim  ( T_mN\cap \D|_m)\] 
		being constant as a function of $m\in N$. Hence, we may cover $N$ by  open subsets $U\subset M$ 
		such that $\D(U)$ is generated by $X_1,\ldots,X_p,Y_1,\ldots,Y_q$, where the $Y_i$'s are linearly independent 
		vector fields spanning 	a complement to $TN|_m\cap \D|_m$ in $\D|_m$ at points $m\in U\cap N$, while 
		$X_1,\ldots,X_p$ are tangent to $N$. The restrictions of $X_i$'s  
		to $U\cap N$ are then generators of $(\si^!\D)(U\cap N)$. 
		%such that there are  linearly independent vector fields $X_1,\ldots,X_q\in \D(U)$ whose values at points $m\in U\cap N$ span a complement to $ Let $E\to U\cap N$ be the subbundle of $TM|_{U\cap N}$ spanned by these vector fields. Then 	$T_mN\oplus E_m=T_mN+\D|_m$ for all $m\in U\cap N$. Let $R$ be a subbundle complementary to $T(U\cap N)\oplus E$, so that 		\[ TM|_{U\cap N}=T(U\cap N)\oplus E\oplus R.\]	The map $\mf{X}(U)\to \mf{X}(U\cap N)$ given by restriction (as a section of the tangent bundle) $X\mapsto X|_{U\cap N}$ followed by projection along $E\oplus R$ restricts to a surjective map 	$\D(U)\to (\si^!\D)(U\cap N)$. In particular, a finite set of generators of $\D(U)$ projects to a finite set of generators for $(\si^!\D)(U\cap N)$. 
		The last claim follows since relatedness of vector fields with respect to smooth maps is compatible with Lie brackets. 
	\end{proof}
	Restriction to submanifolds can be iterated: suppose $N'\subset N\subset M$ are nested submanifolds, with inclusions denoted \[ \si\colon N\to M, \ \ \ 	\sj\colon N'\to N,\ \ \ \si'=\si\circ j\colon N'\to M.\] 
	\begin{lemma}
	If $\D$ is a singular distribution, and $\si$ is $\D$-clean, then 
	$\sj$ is $\si^!\D$-clean  if and only if $\si'$ is $\D$-clean,  and in this case 
	\begin{equation}\label{eq:nested} \sj^!\si^!\D=(\si')^!\D.\end{equation} 
   \end{lemma}
\begin{proof}
Suppose $\si$ is $\D$-clean, so that $\dim(\D|_m)-\dim  ( T_mN\cap \D|_m)$ is constant as a function of $m\in N$.
Then $\si^!\D$ is defined, with $(\si^!\D)|_m=T_mN\cap \D|_m$. For $m\in N'$, we have 
\[\dim((\si^!\D)|_m)-\dim (T_mN'\cap (\si^!\D)|_m)
=\dim(T_mN\cap \D|_m )-\dim (T_mN'\cap  \D|_m ),\]
which is constant as a function of $m\in N'$ if and only if $\dim(\D|_m)-\dim (T_mN'\cap  \D|_m )$ 
 is constant as a function of $m\in N'$. This proves the first claim, and \eqref{eq:nested} is a set-theoretic verification. 
\end{proof}

	%We say that $\si\colon N\hra M$ has clean intersection with a singular Lie filtration $\H_\bullet$ on $M$ if 	the intersection with all $\H_{-i}$ is clean. In this case, a pull-back Lie filtration $(\si^!\H)_\bullet$ on $N$ is defined. 
	\smallskip

	We generalize the `restriction to submanifolds' from Lemma \ref{lem:clean} to a pullback operation  under more general maps $\varphi\colon {M'}\to M$, by the usual trick of replacing the map with its graph
	\[ \on{gr}(\varphi)=\{	(\varphi(x),x)|\ x\in {M'}\}\subset M\times {M'}.\]
	Identify $\on{gr}(\varphi)\cong {M'}$ under projection to the second factor, and 
	denote the inclusion map by $\si_{\gr(\varphi)}\colon M'\cong \on{gr}(\varphi)\to M\times {M'}$.
	
\begin{definition}Let $M$ be a manifold with a  singular distribution $\D\subset \mf{X}_M$. 
	We say that $\varphi\colon {M'}\to M$ is  \emph{$\D$-clean} (resp., \emph{$\D$-transverse})  if the dimension of 
	\[ \on{ran}(T_x\varphi)+\D|_{\varphi(x)},\ \ x\in M'\]
	is constant (resp., if $ \on{ran}(T_x\varphi)+\D|_{\varphi(x)}=T_{\varphi(x)}M$). 
\end{definition}

 Note that $\varphi$ is $\D$-clean (resp., transverse)  if and only if $\on{gr}(\varphi)$ is $\D\times \mf{X}_{M'}$-clean (resp., transverse).  We define  
	\[ \varphi^!\D=\si_{\gr(\varphi)}^!(\D\times \mf{X}_{M'}).\]

	\begin{lemma}
		If $\varphi$ is an embedding $\si\colon N\to M$ of a $\D$-clean submanifold, then 
		this definition of pullback agrees with the restriction to $N$, as defined above. 
	\end{lemma}
	\begin{proof}
		As in the proof of Lemma \ref{lem:clean}, $N$ may be covered by open subsets $U$ such that there are generators 
		$X_1,\ldots,X_p,Y_1,\ldots,Y_q\in \D(U)$, with the $X_i$'s tangent to $N$. 
		The sets of the form $U\times (U\cap N)\subset M\times N$ cover the graph $\gr(\si)$. Letting $Z_1,\ldots,Z_r\in \mf{X}(U\cap N)$ be generators for the module of vector fields on $U\cap N$, the vector fields 
		of the form $(X_i,X_i|_{U\cap N})$, $(Y_j,0)$ and $(0,Z_k)$ are generators of $(\D\times \mf{X}_N)
		( U\times (U\cap N))$, with the 
		$ (X_i,X_i|_{U\cap N})$'s tangent to $\gr(\si)$. The restrictions 
		$ (X_i,X_i|_{U\cap N})|_{\gr(\si)}$
		are just $X_i|_{U\cap N}$, under the  identification  $\gr(\si)\cong N$. 
	\end{proof}
	
	As a special case, the clean intersection condition is automatic when $\varphi\colon M'\to M$ is a submersion. In this case, $\varphi^!\D$ is locally generated by vector fields on $M'$ that are $\varphi$-related to vector fields on $M$. In particular, if $\varphi$ is a local diffeomorphism, then $\varphi^!\D$ is the obvious pullback. Let us finally remark that the pullback operation is well-behaved under composition: 
	\begin{lemma}
		Let $\D$ be a singular distribution on $M$, and suppose the smooth map $\varphi\colon M'\to M$ is $\D$-clean. 
		Let 
		$\psi\colon M''\to M'$ be another smooth map.  Then $\psi$ is  $\varphi^!\D$-clean 
		 if and only if $\varphi\circ \psi$ is $\D$-clean, and in this case $(\varphi\circ \psi)^!\D=\psi^!\varphi^!\D$.
	\end{lemma}
	\begin{proof}
		The embeddings of graphs $\psi_1=\si_{\gr(\psi)},\ \varphi_1=\si_{\gr(\varphi)},\ (\varphi\circ\psi)_1=\si_{\gr(\varphi\circ \psi)}$ 
		fit into a commutative diagram 
		\[ \xymatrix{M'' \ar[r]^{(\varphi\circ\psi)_1} \ar[d]_{\psi_1} & M\times M'' \ar[d]^{\on{id}_M\times \psi_1}\\
			M'\times M''\ar[r]_{\varphi_1\times \on{id}_{M''}} & M\times M'\times M''	
		}
		\] 
		The composed map is the embedding 
		\[ (\varphi\circ \psi)\times \psi\times \on{id}_{M''}\colon M''\to M\times M'\times M''.\]
		We have 
		\[(\varphi\circ \psi)_1^! (\on{id}_M\times \psi_1)^!(\D\times \mf{X}_{M'}\times \mf{X}_{M''})=(\varphi\circ \psi)_1^!(\D\times \mf{X}_{M''})
		=(\varphi\circ \psi)^!\D,\]
		and similarly 
		\[ 	\psi_1^!(\varphi_1\times \on{id}_{M''})^!(\D\times \mf{X}_{M'}\times \mf{X}_{M''})
		=\psi_1^!(\varphi^!(\D)\times \mf{X}_{M''})=\psi^!\varphi^!\D.\]
		On the other hand, by \eqref{eq:nested} each of these coincide with 
		\[	\big((\varphi\circ \psi)\times \psi\times \on{id}_{M''}\big)^!(\D\times \mf{X}_{M'}\times \mf{X}_{M''}). \qedhere\]
	\end{proof}

\subsection{Singular Lie filtrations}
The concept of singular Lie filtration generalizes singular foliations, as well as (regular) Lie filtrations.
%For any vector subbundle $K\subset TM$, its sheaf of sections is locally finitely generated. 

\begin{definition}
A \emph{singular Lie filtration} of order $r$ is a filtration of the sheaf of vector fields
		\begin{equation}\label{eq:liefiltration} \mf{X}_M=\H_{-r}\supset \H_{-r+1}\supset \cdots \supset \H_0\supset  0,\end{equation}
by singular distributions (i.e., locally finitely generated $C^\infty_M$-submodules)  	such that 
	\[  [\H_{-i},\H_{-j}]\subset \H_{-i-j}\] 
	for all $i,j$. It is called a \emph{(regular) Lie filtration}
	if the $\H_{-i}$ are sheaves of sections of subbundles $H_{-i}\to M$ of the tangent bundle. 
\end{definition}

\begin{remark}
	Note that we are allowing for a non-trivial $\H_0$. 
	The bracket condition then shows that $\H_0$ is involutive, and so defines a singular foliation of $M$. We shall see that leaves of this singular foliation acquire natural weightings. On the other hand, we will construct weightings along more general submanifolds with a `clean intersection' property. Since that construction does not involve the summand $\H_0$, we will put $\H_0=0$ in the next section. 
\end{remark}

\subsection{Examples}
	Regular Lie filtrations have been much studied in recent years as a framework for the theory of  hypo-elliptic operators. See the work of Choi-Ponge \cite{choi:priv2,choi:priv1,choi:tan}, 
	van Erp-Yuncken \cite{erp:tan}, Haj-Higson \cite{hig:eul}, Dave-Haller \cite{dav:bgg,dav:hea},  Mohsen \cite{moh:ind,moh:def}, among others. These references provide many examples; see \cite{choi:priv1} for an overview. 
	The singular Lie filtrations play a similar role for a broader class of hypo-elliptic operators \cite{and:inprep}. Other examples arise,  for instance, in the context of sub-Riemannian geometry.

\begin{example}
	A \emph{Carnot manifold} (also called a Carnot-Carath{\'e}odory manifold) is a manifold $M$ with a subbundle $D\subset TM$, 
	with sheaf of sections $\D\subset \X_M$, such that iterated brackets of $\D$ generate all of $\X_M$. 
	See, e.g., \cite{pan:sub}. One obtains a singular Lie filtration by letting 
	\[ \H_{-1}=\D,\] and inductively 
	\begin{equation}\label{eq:inductive} \H_{-i-1}=\H_{-i}+[\D,\H_{-i}].\end{equation}
	The Carnot manifold is called \emph{equiregular} if this is a regular Lie filtration. An example of an equiregular Carnot manifold is given by 
\[  D=\on{span}\big\{\f{\p}{\p x},\ \f{\p}{\p y}+x\f{\p}{\p z}\big\}\subset T\R^3.\]
	An example of a 
	Carnot manifold that is not equiregular is given by the Martinez Carnot structure
	\[ D=\on{span}\big\{\f{\p}{\p x},\ \f{\p}{\p y}+x^2\f{\p}{\p z}\big\}\subset T\R^3.\]
	%Here\[ \H_{-1}(M)=\on{span}_{C^\infty(M)}\big\{\f{\p}{\p x},\ \f{\p}{\p y}+x^2\f{\p}{\p z}\big\},\ \ \H_{-2}(M)=\H_{-1}(M)+\on{span}_{C^\infty(M)}\big\{x\f{\p}{\p z}\big\}.\]
\end{example}

	\begin{example}
	Every singular distribution $\D\subset \mf{X}_M$ defines a singular Lie filtration of order $r=2$, 
	\[ \mf{X}_M=\H_{-2}\supset \D=\H_{-1}.\]
	If $\D$ is a regular distribution, then this is a regular Lie filtration. 
	More generally, 
	one obtains a singular Lie filtration of order $r$  by putting $\H_{-1}=\D,\ \H_{-r}=\mf{X}_M$, and 
	using 
		\eqref{eq:inductive} for $-i-1\ge -r$. In particular, any finite collection of vector fields defines a singular Lie filtration, by taking $\D$ to be the submodule spanned by them.  
\end{example}

\begin{example}\label{ex:liegroup}
	Let $G$ be a Lie group whose Lie algebra $\g$ has a filtration $\g=\g_{-r}\supset \cdots \supset \g_0$, i.e, 
	$[\g_{-i},\g_{-j}]\subset \g_{-i-j}$. Using left-translation, the filtration of the Lie algebra defines a (right-invariant) regular Lie filtration of $TG=G\times \g$	by subbundles $H_{-i}=G\times \g_{-i}$. If $\g_0$ exponentiates to a closed subgroup 
	$H$, then the foliation defined by $\g_0$ has as its leaves the right-translates of $H$. 
\end{example}

\begin{example}\label{ex:liegroupoid}
Generalizing Example \ref{ex:liegroup}, let $G\rra M$ be a Lie groupoid. Let $s,t\colon G\to M$ be the source, target maps.
Suppose the Lie algebroid of $G$ has a bracket-compatible filtration $\g=\g_{-r}\supset \cdots \supset \g_0$ by subbundles. Let $H_{-i}\subset TG$ be the subbundle spanned by left-invariant vector fields (tangent to $t$-fibers) 
$\xi^L$ for $\xi\in \g_{-i}$. Then $H_0,\ldots,H_{-r}=\ker(T t)$ together with         $H_{-r-1}=TG$ defines a regular Lie filtration
of order $r+1$.
\end{example}

\begin{example}\label{ex:weightinggivesliefiltration}
	Given a weighting of order $r$ along a submanifold $N\subset M$, with the resulting filtration on vector fields, 
	we obtain a singular Lie filtration of order $r$ by truncation. We shall use the special notation 
\[ \K_{-i}=\mf{X}_{M,(-i)},\ \ \ i=0,\ldots,r.\] 
	In local weighted coordinates, the module $\K_{-i}(U)$ is generated by all $x^s\f{\p}{\p x_a}$  such that $s\cdot w\ge w_a -i$. We also have the singular Lie filtration of order $r+1$ by the submodules 	$\K_{-i}^N$ together with 
	$\mf{X}_M$ in degree $-r-1$. 
\end{example}
\medskip

The basic constructions for singular distributions give corresponding constructions for singular Lie filtrations. In particular, 
products $(\H'\times \H'')_\bullet$ of singular distributions are defined by setting $(\H'\times \H'')_{-i}=\H'_{-i}\times \H''_{-i}$. Pullbacks $\varphi^!\H_\bullet$ of singular Lie filtrations under smooth maps $\varphi\colon M'\to M$
are defined provided that  $\varphi$ is \emph{$\H_\bullet$-clean}, that is, it is 
$\H_{-i}$-clean for all $i$. The compatibility with brackets is clear in the case of embeddings, and for general maps $\varphi$ follows by turning the map into an embedding.

In the following section we will show that $\H_\bullet$-clean submanifolds define weightings. 
Here are some examples of such submanifolds. 

\begin{examples}
	Let $M$ be a manifold with a singular Lie filtration $\H_\bullet$. 
\begin{enumerate}
	\item Every point $N=\{m\}$ is  $\H_\bullet$-clean.  
	\item Suppose $\F\subset \mf{X}_M$ is a singular foliation with the property $[\F,\H_{-i}]\subset \H_{-i}$ for all $i$.  Then the leaves $N$ of $\F$ are $\H_\bullet$-clean submanifolds. Indeed, since the vector fields in $\F(U)$ act by infinitesimal automorphisms of $\H_{-i}(U)$, the dimensions of $T_mN+\H_{-i}|_m$ are constant along $N$. 
	\item Given a Lie group action  on $M$, preserving the singular Lie filtration, the cleanness condition holds  
	true for all orbits of the action. 
	\item Let $i$ be the smallest index for which $\H_{-i}\neq 0$. 	
	If $N$ is $\H_{-i}$-transverse, then it is $\H_{-i'}$-transverse for all $i'\ge i$, and in particular is $\H_\bullet$-clean. 
	\item If $\H_\bullet$ is a singular Lie filtration on $M$, then the diagonal $\Delta_M\subset M\times M$ is 
		$\H_\bullet\times \H_\bullet$-clean if and only if $\H_\bullet$ is a regular Lie filtration. 
	\item If $\H_\bullet$ is a regular Lie filtration, given as the sheaves of sections of a sequence of subbundles $H_{-i}\subset TM$, the cleanness condition means that $H_{-i}\cap TN$ are subbundles of $TN$.	
	\item In particular, in Example \ref{ex:liegroupoid} the unit space $M \subset G$ satisfies the cleanness condition.
\end{enumerate}	
\end{examples}

\section{Weightings from singular Lie filtrations}\label{sec:weightingalongleaves}
 Throughout this section, $M$ is a manifold with a singular Lie filtration $\H_\bullet$. Our construction of weightings will not involve $\H_0$, hence we will assume throughout this section that $\H_0=0$: 
\begin{equation}\label{eq:singularlie} \mf{X}_M=\H_{-r}\supset \cdots \supset \H_{-1}\supset 0.\end{equation}

\subsection{Construction of weighting}
Suppose  $N\subset M$ is an $\H_\bullet$-clean closed submanifold. We obtain a filtration 
\begin{equation}\label{eq:filtration_tm} TM|_N=\ti{F}_{-r}\supset \cdots \supset \ti{F}_{0}=TN\end{equation} 
by subbundles $\wt{F}_{-i}$, where 
$\wt{F}_{-i}|_m=T_mN+\H_{-i}|_m$, and a resulting filtration of the normal bundle 
\begin{equation}\label{eq:filtration_nu} \nu(M,N)=F_{-r}\supset \cdots \supset F_0=0.\end{equation} 
Define a filtration
\[ C^\infty_M=C^\infty_{M,(0)}\supset C^\infty_{M,(1)}\supset \cdots \]
by induction, starting with $C^\infty_{M,(1)}=\I$, where $\I$ is the vanishing ideal of $N$, and for  $i>1$ 
\[ C^\infty(U)_{(i)}=\{f\in C^\infty(U)|\  \forall X\in \H_{-j}(U),\ 0<j<i \colon\  \L_X f\in C^\infty(U)_{(i-j)}\}.\]
(note that the condition on $\L_X f$ would be vacuous if $j\ge i$). 
From the definition, it is clear that the filtration is multiplicative: $C^\infty_{M,(i_1)}\cdot C^\infty_{M,(i_2)}\subset C^\infty_{M,(i_1+i_2)}$.

\begin{theorem}\label{th:mainresult}
	Let $M$ be a manifold with a singular Lie filtration $\H_\bullet$ of order $r$, and suppose 
	$N\subset M$ is an $\H_{\bullet}$-clean submanifold. Then the filtration of $C^\infty_M$ described above 
	is a weighting of order $r$ along $N$, with \eqref{eq:filtration_nu} as the associated filtration of the normal bundle. 
\end{theorem}

The proof of this result is by construction of local weighted coordinates. For the case of a regular Lie filtration, this was done by Choi-Ponge \cite{choi:priv1} in the case of $\dim N=0$, and by Haj-Higson \cite{hig:eul} for $\dim N>0$.

\subsection{Proof of Theorem \ref{th:mainresult}}
Given the assumptions from Theorem \ref{th:mainresult}, we will produce weighted coordinates near any given point $m\in N$. The argument is similar to 
a proof in \cite{loi:wei}, which, in turn, builds on the constructions 
of \cite{bel:tan,choi:priv1,choi:priv2}. It  will require several steps. \smallskip

We first note that the filtration $\H_\bullet$ on the sheaf of vector fields determines a filtration on differential operators, 
\[                \cdots \supset      \on{DO}_{M,-2}  \supset  \on{DO}_{M,-1}\supset \on{DO}_{M,0}.\]
Here,  $\on{DO}_{0}(U)=C^\infty(U)$, while $\on{DO}_{-j}(U)$ for $j>0$ is spanned by sums 
of products $X_1\cdots X_k$ with $X_\nu\in \H_{-j_{\nu}}$ and $j_1+\ldots+j_k\ge j$. We say that 
$D\in \on{DO}_{-j}(U)$ has $\H$-weight $-j$. 

\begin{remark}
	The filtration on $C^\infty_M$ determines another filtration on differential operators, which depends on the choice of $N$, and which is usually  different from the filtration by $\H$-weight. 
\end{remark}
The filtration on $C^\infty_M$ can now be rephrased as follows: $f\in C^\infty(U)_{(i)}$ if and only if 
\begin{equation} j<i,\ D\in \on{DO}_{-j}(U)\ \Rightarrow \ Df|_N=0.\end{equation}
Let 
\[ k_0=\dim N,\ k_1=\dim \tilde{F}_{-1},\ldots,\ k_{r}=\dim \tilde{F}_{-r}=n,\] 
and let $w_1,\ldots,w_n$ be the corresponding weight sequence, so that $w_a=i$ for $k_{i-1}<a\le k_i$. 
Choose an open neighborhood $U$ of the given point $m$ and linearly independent vector fields
\[ V_a\in\mf{X}(U),\ \ a=k_0+1,\ldots,n\] 
such that for all $j>0$, 
the vector fields $V_{k_0+1},\ldots,V_{k_j}$ are in $\H_{-j}(U)$, and represent a frame for 
$F_{-j}|_{U\cap N}$. 
Given a multi-index $s=(s_{k_0+1},\ldots,s_n)$ with $s_a\ge 0$, let 
\[ V^s=\prod_a V_a^{s_a}=
V_{k_0+1}^{s_{k_0+1}}\cdots V_n^{s_n}\]
be the corresponding differential operator of order $|s|=\sum_a s_a$ and $\H$-weight $-j$, where $j=s\cdot w=\sum_a s_a w_a$.  

\begin{lemma}
	A function $f\in C^\infty(U)$ has filtration degree $i$ if and only if \begin{equation}\label{eq:condition}(V^s f)|_N=0\end{equation}
	for all multi-indices $s$ with $s\cdot w<i$. 
\end{lemma}
\begin{proof}
Clearly, if $f$ has filtration degree $i$, then the condition \eqref{eq:condition} holds since $V^s$ is a differential operator of $\H$-weight $-j$, with 
$j<i$. For the converse, suppose the condition \eqref{eq:condition} is satisfied. 
We want to show that $Df|_N=0$ for all differential operators $D$ of $\H$-weight $-j$  with $j<i$.  
Using induction on the order $k$ of differential operators, we may assume that this holds true  for all such differential operators of order less then a given number $k$. To prove it for differential operators $D$ of order $k$, it suffices to show that any $D\in \on{DO}^{k}(U)_{-j}$ may be written in the form 
\begin{equation}\label{eq:rearranged} D=\sum_s f_s V^s +D'+D''\end{equation}
where the sum is over multi-indices $s$ with $|s|=k$ and $s\cdot w=j$, where $D'\in \on{DO}^{k}(U)_{-j}$ is a 
sum of products $Y_1\cdots Y_k$ (with $Y_\nu\in \H_{-\ell_\nu}(U),\ \sum \ell_\nu=j$) such that   
the first vector field $Y_1$ is tangent to  $N$, and where $D''\in \on{DO}^{k-1}(U)_{-j}$. Once this is shown, 
we have $D''f|_N=0$ by induction hypothesis,  and similarly $D'f|_N=0$ since $Y_1\cdots Y_k f|_N=(Y_1|_N)(Y_2\cdots Y_k f|_N)=0$.

The decomposition \eqref{eq:rearranged}  follows from the following observations: 
\begin{enumerate}
	\item Suppose $X_1,\ldots,X_k$ are vector fields with $X_\nu\in \H_{-j_\nu}(U),\ \  j_1+\ldots+j_k=j<i$. Then 
	\[  (X_1\cdots X_\nu\cdots X_{\nu'}\cdots X_k)- (X_1\cdots X_{\nu'}\cdots X_{\nu}\cdots X_k)\in \on{DO}^{k-1}(U)_{-j},\]
	for all $\nu\neq \nu'$. Hence, modulo differential operators of lower order
	we may re-order the $X_\nu$ as we please.  
	In particular, if any of the $X_\nu$ is tangent to $N$, we may `move it to first place' .
	\item 
	Similarly, given $g\in C^\infty(U)$, we have 
	\[  (X_1\cdots (g X_\nu)\cdots X_k)-g\,( X_1\cdots X_\nu\cdots X_k)\in \on{DO}^{k-1}(U)_{-j}.\]
	Since any $X\in \H_{-\ell}(U)$ is of the form $X=X'+\sum f_a V_a$ with $X'$ tangent to $N$, $f_a\in C^\infty(U)$,
	and $V_a\in \H_{-\ell}(U)$, we may use this to re-arrange any product of $X_\nu$'s in the form \eqref{eq:rearranged}. 
\end{enumerate}
 
\end{proof}

We now proceed as in \cite{loi:wei}. Taking $U$ smaller if needed, choose
coordinates $x_1,\ldots x_n$ on $U$ such that 
\begin{equation}\label{eq:start} V_a(x_b)|_N=\delta_{ab},\ \ a>k_0. \end{equation}
(In particular, $x_1,\ldots,x_{k_0}$ restrict to coordinates on $U\cap N$.) We will show how to modify the coordinates in 
such a way that $x_a$ has weight $w_a$ (while retaining the property  \eqref{eq:start}). For $w_a\le 2$, no modification is needed. Indeed, the coordinates $x_{k_0+1},\ldots,x_{k_1}$ have weight $1$ since they vanish on $N$, while 
$x_{k_1+1},\ldots,x_{k_2}$ have weight $2$ since their differentials vanish on $\wt{F}_{-1}$. 
However, the coordinates $x_a$ with $w_a\ge 3$ may require adjustment. Suppose by induction that for a given $\ell\ge 2$, 
the coordinates $x_a$ with $k_{\ell-1}<a\le k_\ell$ have weight $\ell$. For $x_a$ with $k_\ell<a\le k_{\ell+1}$,   
we  look for a coordinate change of the form  
\[ \wt{x}_a=x_a+\sum  \chi_{au}\,  x^u \]
(using multi-index notation  $x^u=x_1^{u_1}\cdots x_n^{u_n}$), 
where the sum is over multi-indices with  
\[ |u|=\sum_b u_b\ge 2,\ \ \ \  w\cdot u\,<\,w_a,\ \ \ \ u_b=0 \mbox{ for } b\le k_0\]
such that the coefficients $\chi_{au}\in C^\infty(U)$ depend only on the coordinates $x_1,\ldots,x_{k_0}$. 
The condition $|u|\ge 2$ means that $\sum_u  \chi_{au}\,  x^u\in \I^2(U)$; hence the coordinate change will retain the property \eqref{eq:start}. The property $w\cdot u\,<\,w_a$ means, in particular, that only $x_b$'s with $b\le k_\ell$ enter the expression for $\sum_u  \chi_{au}\,  x^u$. 
The coordinate function $\wt{x}_a$ has filtration degree $w_a$ if and only if 
\[ (V^s \wt{x}_a)|_N=0\] 
for all multi-indices $s=(s_{k_0+1},\ldots,s_n)$ with $w\cdot s<w_a$. As explained in \cite{loi:wei}, these conditions 
on the functions $\chi_{as}$ have a unique solution, defined recursively in terms of $|s|$: 
\[ \chi_{as}=-\f{1}{c_s}\Big(V^s x_a|_N+\sum_{u\colon 2\le |u|<|s|} \big(V^s(\chi_{au}\,  x^u)\big)\big|_N\Big), \qquad c_s=V^s x^s|_N.\]
In conclusion, with this choice of $\chi_{as}$ the new coordinates $\wt{x}_a$ have weight $\ell+1$. Rename $\wt{x}_a$ as $x_a$, and proceed. 
The conditions $V_a(x_b)=\delta_{ab}$ for 
$a>k_0$ show that the filtration of $TM|_{U\cap N}$ for this weighting is given by the subbundles spanned by 
\[ TN+\on{span}\{V_a|_N,\ k_0<a\le k_i\}=\wt{F}_{-i}|_{U\cap N}.\]

\subsection{Examples}
Here are two examples illustrating the construction of weighted coordinates for singular Lie filtrations.

\begin{example}
Consider the vector fields 
\[ X=\f{\p}{\p x}+x\f{\p}{\p z},\ \ Y=\f{\p}{\p y},\ \ Z=\f{\p}{\p z}\]
on $M=\R^3$. Define a regular Lie filtration of order $3$, where $\H_{-1}$ is spanned by $X$, $\H_{-2}$ is spanned by $X,Y$, and 
$\H_{-3}=\mf{X}_M$. 
This defines an order $3$ weighting at $N=\{0\}$,with $w_1=1,\,w_2=2,\,w_3=3$. Take $V_1=X,\ V_2=Y,\ V_3=Z$ 
to be the frame of the discussion above. The original coordinates $x_1=x,\,x_2=y,\,x_3=z$  satisfy 
$V_a x_b|_0=\delta_{ab}$, but they are not weighted coordinates since $z$ does not have weight $3$. To obtain 
weighted coordinates, we use a coordinate change $\wt{z}=z+\lambda x^2$. This satisfies 
\[\L_X\wt{z}=x+2\lambda x,\]
which has weight $3-1=2$ if and only if $\lambda=-\f{1}{2}$. We  conclude that
\[ x,\ y,\ z-\hh x^2\]
is the desired set of weighted coordinates. Note that in the new coordinates, $X,Y,Z$ are just the coordinate vector fields. 
\end{example}

\begin{example}
For a singular Lie filtration that is not regular, consider the following example of a 
Martinez-Carnot structure on $M=\R^3$: Let
	\[X=\f{\p}{\p x}+(2x+y)\f{\p}{\p z},\ Y=\f{\p}{\p y}+(x+x^2)\f{\p}{\p z}.\]
Define a singular Lie filtration of order 4, where 	
$\H_{-1}$ is spanned by $X$, $\H_{-2}$ is spanned by $X,Y$, 
$\H_{-3}$ %=\H_{-2}+[\H_{-1},\H_{-2}]$ 
 is spanned by $X,Y,[X,Y]$, and $\H_{-4}=\mf{X}_M$.

%\[ \H_{-3}=\on{span}\{\f{\p}{\p x}+(2x+y)\f{\p}{\p z},\f{\p}{\p y},x\f{\p}{\p z}\},\]
Again, let $N=\{0\}$. Take $V_1=X,\ V_2=Y,\ V_3=Z=\f{\p}{\p z}$ corresponding to $w_1=1,w_2=2,w_3=4$. 
The coordinates $x,y$ have filtration degrees $1,2$ has required, but $z$ does not have filtration degree $4$. To obtain weighted coordinates, we seek a coordinate change of the form $\wt{z}=z+\lambda x^2+\mu xy$. From 
\[ \L_X\wt{z}=(2x+y)+2\lambda x+\mu y,\ \ \ \L_Y\wt{z}=(x+x^2)+\mu x\]
we see that $\L_Y\wt{z}$ has filtration $2$ if and only if $\mu=-1$, and 
$\L_X \wt{z}$ has filtration degree $3$ if and only if furthermore $\lambda=-1$. Hence, 
\[ x,\ y,\ z-x^2-xy\]
is the desired set of weighted coordinates. 
\end{example}

\section{Singular Lie filtrations in terms of higher tangent bundles}\label{sec:trm}
In \cite{loi:wei}, we gave an alternative description of weightings on $M$ in terms of subbundles $Q$ of higher tangent bundles $T_rM$.  In this section, we will explain that similarly, singular Lie filtrations admit descriptions as singular foliations of higher tangent bundles. Given an $\H$-clean submanifold $N\subset M$, the subbundle $Q$ for the corresponding weighting along $N$ is described as the flow-out of $T_rN\subset T_rM$. 
In this section, we will temporarily abandon the sheaf language, for notational convenience.

\subsection{Higher tangent bundles} \label{subsec:trm}
We begin with some background material on higher tangent bundles. A reference for some of this material is the book \cite{kol:nat}. 

\smallskip
The \emph{$r$-th tangent bundle} $T_rM\to M$, also known as \emph{bundle of $r$-velocities}, was introduced by Ehresmann as the space of $r$-jets of curves 
\[ T_rM=J_0^r(\R,M).\] Its elements are equivalence classes of curves $\gamma\colon \R\to M$, where $\gamma_1,\gamma_2$ are considered equivalent if $\gamma_1(0)=\gamma_2(0)$ and the Taylor expansions 
of the two curves in a coordinate chart  agree up to order $r$. There is also an algebraic definition, which for 
us will be more convenient: Let $\AA_r$ be the unital algebra with a single generator $\epsilon$ and relation $\epsilon^{r+1}=0$. Then 
\begin{equation}\label{eq:trm} T_rM=\Hom_\alg(C^\infty(M),\AA_r).\end{equation}
Elements of $T_rM$ are sums $\su=\sum_{i=0}^r \su_i \epsilon^i$ with $\su_i\colon C^\infty(M)\to \R$, where 
$\su_0$ is an algebra morphism (specifying a base point in $M$), $\su_1$ is a  derivation with respect to $\su_0$ (specifying a tangent vector), and so on. The smooth structure on $T_rM$  is characterized by the property  that for all 
$f\in C^\infty(M)$, the function given by evaluation 
\[T_rf\colon T_rM\to \AA_r,\ \ (T_rf)(\su)=\su(f) \] 
is again smooth. For $r>1$, the $r$-th tangent bundle is not a vector bundle, but is a graded bundle (see Section \ref{subsec:graded}), with the monoid action of $t\in \R$ 
given by the algebra morphism of $\AA_r$ taking $\sum_{i=0}^r \su_i \epsilon^i$ to $ \sum_{i=0}^r \su_i t^i\epsilon^i$. 
The $r$-th tangent bundle fits into a tower of fiber bundles 
\begin{equation}\label{eq:tower} \cdots \to T_rM\to T_{r-1}M\to \cdots \to TM\to M\end{equation}
where the maps $T_rM\to T_{r-1}M$ are induced by the algebra morphisms $\AA_r\to \AA_{r-1}$. 
The tangent bundle $TM\to M$ (regarded as a Lie group bundle) acts on $T_rM$ by 
\begin{equation} 
\label{eq:TMaction}
TM\times_{M}T_rM\to T_rM,\ \ v\cdot \su=\su-v\epsilon^r;
\end{equation}
the maps in  \eqref{eq:tower} may also be seen as the quotient maps for this action. 

\begin{remark}
The tangent bundle $TM$ may be identified with the normal bundle of the diagonal in $M\times M$. Similarly, 
$T_rM$ may be identified with the weighted normal bundle of the diagonal in $M^{r+1}$, for a suitable weighting. 
Details will be given elsewhere. 	
\end{remark}

\subsection{Lifts}
For $f\in C^\infty(M)$ we denote by $f^{(i)}\in C^\infty(T_rM)$ the components of $T_rf$, so that 
\[ T_rf=\sum_{i=0}^r f^{(i)}\epsilon^i.\]
Here $f^{(0)}$ is the pullback of $f$ under the base projection, $f^{(1)}$ is the pullback of the exterior differential $\d f\in C^\infty(T_1M)$ under the map $T_rM\to TM$; more generally, $f^{(i)}$ is the pullback of a function on $T_iM$. The function 
$f^{(i)}$ is homogeneous of degree $i$ for the scalar multiplication on $T_rM$. 
%\begin{remark}	Note that $\AA_r=J_0^r(\R,\R)=T_r\R$, with the algebra multiplication structure induced 	by the product of $\R$. Then $T_rf$ is just the $r$-th prolongation $T_rM\to T_r\R$. 	\end{remark}
 The \emph{tangent lift} 
\[ T_rX\in \mf{X}(T_rM)\]
of a vector field $X\in \mf{X}(M)$ is characterized by the property $(T_rX)(T_rf)=T_r(Xf)$. We also use the notation $X^{(0)}=T_rX$, so that  $X^{(0)}f^{(i)}=(Xf)^{(i)}$.  The \emph{vertical lifts} 
\[ X^{(-1)},\ldots,X^{(-r)}\] are similarly defined by $X^{(-j)}f^{(i)}=(Xf)^{(i-j)}$; the fact that these vanish on all $f^{(0)}$ 
implies that they are  tangent to the fibers of $T_rM\to M$ everywhere. 
%(Indeed, $X^{(-j)}$ vanishes on all lifts of the form $f^{(0)},\ldots,f^{(j-1)}$ and hence is tangent to the fibers of $T_rM\to T_{j-1}M$.)  
The superscript indicates the homogeneity, i.e., $\kappa_t^* X^{(-j)}=t^{-j} X^{(-j)}$ for $t\neq 0$. The lifts satisfy 
\begin{equation}\label{eqLiftproperties} [X^{(-i)},Y^{(-j)}]=[X,Y]^{(-i-j)},\ \ (fX)^{(-i)}=\sum_{j=0}^{r-i} f^{(j)}X^{(-i-j)}.\end{equation}
The vector fields $X^{(-r)}$ define a vector bundle action of $TM\to M$ on $T_rM\to M$ (as in \eqref{eq:TMaction}), with quotient $T_{r-1}M$. 
If $x_a$ for $ a=1,\ldots,n$ are local coordinates on $U\subset M$, then the functions 
\[ x_a^{(i)},\ \ 1\le a\le n,\ \ 0\le i\le r\] serve as fiber bundle coordinates on $T_rU\subset T_rM$. 
The tangent lift of $X=\sum_a f_a \f{\p}{\p x_a}$ is 
\[ X^{(0)}=\sum_{a,i} f_a^{(i)} \f{\p}{\p x_a^{(i)}}.\]
The lift $X^{(-j)}$ is obtained from this expression by replacing $\f{\p}{\p x_a^{(i)}}$ with $\f{\p}{\p x_a^{(i+j)}}$ if 
$i+j\le r$, with $0$ otherwise.

We shall also need the following observation, discussed in the articles \cite{mor:pro,oka:pro} where it is 
attributed to Koszul.
\begin{proposition}[Koszul]
There is a natural action 	
\[ \AA_r\to \Gamma(\on{End}(T(T_rM)))\]
of the algebra 	$\AA_r$ on the  fibers of the tangent bundle of $T_rM$,
in such a way that the generator $\epsilon\in \AA_r$ acts as 
\[ \epsilon\cdot X^{(0)} =X^{(-1)},\ldots,
\epsilon\cdot 
X^{(-r+1)}=X^{(-r)},\ \epsilon\cdot X^{(-r)}=0\]
for all $X\in\mf{X}(M)$. 
\end{proposition}
One way of describing this algebra action uses the identification 
\[ T(T_rM)=\Hom_\alg(C^\infty(M),\AA_1\otimes \AA_r).\]
Elements of $\on{End}_\alg(\AA_1\otimes \AA_r)$ act on $T(T_rM)$ by composition of algebra morphisms. The Koszul action comes from the inclusion $\AA_r\to \on{End}_\alg(\AA_1\otimes \AA_r)$, where 
$x\in \AA_r$ acts as  
\[ x\cdot(1\otimes y+\epsilon\otimes z)=1\otimes y+\epsilon\otimes xz.\]

 \subsection{The group structure on $\Gamma(T_rM)$}

 The group structure on sections of the tangent bundle, 
 given by addition of vector fields, generalizes to a nilpotent group structure on sections of 
 the $r$-th tangent bundle $T_rM\to M$. Similar group structures on sections are discussed in 
 \cite[Chapter 37.6]{kol:nat} in the general context of Weil functors. Likewise the action of $\Gamma(TM)$ on $TM$ generalizes to an action of $\Gamma(T_rM)$ on $T_rM$.
 
We begin with the characterization 
 of diffeomorphisms as algebra automorphisms 
 \[\on{Diff}(M)= \Aut_\alg(C^\infty(M));\]
 here a diffeomorphism $\Phi$ corresponds to the algebra automorphism $\Phi_*$ given by push-forward of functions. Identifying 
 $M=\Hom_\alg(C^\infty(M),\R)$ via evaluation maps $m\mapsto \on{ev}_m$, the action of $\Aut_\alg(C^\infty(M))$ on  $M$ is given by $\on{ev}_m\mapsto \on{ev}_{\Phi(m)}=\on{ev}_m\circ (\Phi_*)^{-1}$. To extend to the r-th tangent bundle
 it is convenient to write \eqref{eq:trm} as
 \[ T_rM=\Hom_{\AA_r-\alg}(C^\infty(M)\otimes\AA_r,\AA_r)\]
 where the subscript indicates $\AA_r$-linear maps. The group 
 \[ \mathfrak{U}_r=\Aut_{\AA_r-\alg}(C^\infty(M)\otimes \AA_r),\]
 acts on $T_rM$ by $U\cdot \su= \su\circ U^{-1}$.

 \begin{remark}
 	\label{rm:jetpicture}
 	In the jet picture, note that on $T_rM$ there is an action of the group of smooth 1-parameter families of diffeomorphisms $\Phi \colon \mathbb{R} \times M \rightarrow M$ of $M$: $\Phi$ sends the equivalence class of the smooth curve $\gamma \colon \mathbb{R} \rightarrow M$ to the equivalence class of the smooth curve $t\mapsto \Phi(t,\gamma(t))$. Then $\mf{U}_r$ is the quotient of this group by the normal subgroup acting trivially on $T_rM$.
 \end{remark}

 We may write the $\AA_r$-module endomorphisms of 
 $C^\infty(M)\otimes \AA_r$ as
 \[ U=\sum_{i=0}^r U_i \epsilon^i,\ \ \ U_i\in \on{End}(C^\infty(M)).\]
 This defines an $\AA_r$-linear  \emph{algebra} endomorphism if and only if 
 \begin{equation}\label{eq:88} U_i(fg)=\sum_{i_1+i_2=i}U_{i_1}(f)U_{i_2}(g),\end{equation}
 and is invertible if and only if $U_0$ is invertible. 
 %The group multiplication on $\mathfrak{U}_r$ reads as
 %\[ U'U=\sum U'_{i_1}\circ U_{i_2}\ \epsilon^{i_1+i_2}.\]
 The monoid $(\R,\cdot)$ acts on $\mathfrak{U}_r$ by group homomorphisms, via $\sum_{i=0}^r U_i \epsilon^i\mapsto 
 \sum_{i=0}^r U_i t^i \epsilon^i$. Note also that the quotient maps $\AA_r\to \AA_{r-1}$ give a tower of groups and surjective group homomorphisms 
 \[ \cdots \to \mathfrak{U}_r\to \mathfrak{U}_{r-1}\to \cdots \to \mathfrak{U}_0=\on{Diff}(M);\]
 The group $\mf{U}_1$ is a semidirect product $\mf{X}(M)\rtimes \on{Diff}(M)$, with $(X,\Phi)$ corresponding to 
 $U=\Phi_*+\epsilon X$. 
 
 \begin{lemma}\label{lem:ker}
The kernel of $\mathfrak{U}_r\to \mathfrak{U}_{r-1}$ is a copy of $\mf{X}(M)$, with group structure given by addition. 
\end{lemma} 
\begin{proof}
Elements of the kernel are of the form 	$\on{id}_{C^\infty(M)}+U_r\epsilon^r$. Here \eqref{eq:88} says that $U_r$ is a derivation of $C^\infty(M)$. 
\end{proof}
The $\mathfrak{U}_r$-action of $T_rM$ determines an action on %polynomial 
functions, via push-forward. In particular, we are interested in the action on lifts $f^{(i)}$.  
\begin{lemma}\label{lem:formula}
For $U\in \mathfrak{U}_r,\ f\in C^\infty(M),\ i=0,\ldots,r$, 
\[ U\cdot f^{(i)}=\sum_{j=0}^r (U_j(f))^{(i-j)}.\]
\end{lemma}
\begin{proof}
For $\su\in T_rM$, 
\[ (U\cdot T_rf)(\su)=(T_rf)(U^{-1}\cdot \su)
=(T_rf)(\su\circ U)=(\su\circ U)(f).\]
Expanding $T_rf=\sum_i f^{(i)}\epsilon^i,\ U=\sum U_j\epsilon^j$ this becomes 
\[ \sum_i (U\cdot f^{(i)})(\su)\ \epsilon^i=\sum_j \su(U_j(f))\epsilon^j=
\sum_{i,j} (U_j(f))^{(i-j)}(\su) \epsilon^{i}.\]
The lemma follows by comparing  coefficients. 
\end{proof}
Let
\[ \on{Lie}(\mathfrak{U}_r)=\on{Der}_{\AA_r-\alg}(C^\infty(M)\otimes \AA_r)\]
be the space of $\AA_r$-linear derivations of the algebra $C^\infty(M)\otimes \AA_r $. 
Writing its elements as $X=\sum_i X_i \epsilon^i$, the condition $X(fg)=X(f)g+fX(g)$ simply says that all $X_i$ are derivations of $C^\infty(M)$. That is, 
\[ \on{Lie}(\mathfrak{U}_r)=\mf{X}(M)\otimes \AA_r.\]
\begin{lemma}
	The action of the Lie algebra $\on{Lie}(\mathfrak{U}_r)$ on $T_rM$ is given by the map  
\[ \varrho\colon \mf{X}(M)\otimes \AA_r	\to \mf{X}(T_rM),\ \ X=\sum_{j=0}^r X_j\epsilon^j\mapsto \sum_{j=0}^r X_j^{(-j)}.\]
\end{lemma}
\begin{proof}
The infinitesimal version of Lemma \ref{lem:formula} shows that 
\[ X\cdot f^{(i)}=\sum_{j=0}^r (X_jf)^{(i-j)}=\sum_{j=0}^r X_j^{(-j)}\cdot f^{(i)}.\]
\end{proof}

Let $\mathfrak{U}_r^-$ be the subgroup of all $U=\sum U_i \epsilon^i$ for which $U_0=\on{id}$. This group is unipotent (its elements satisfy $(U-\on{id})^{r+1}=0$);  
 its Lie algebra 
$\on{Lie}(\mathfrak{U}_r^-)=\mf{X}(M)\otimes \AA_r^-$  consists of all $X=\sum X_i\epsilon^i $ such that $X_0=0$.
Note that the action of $\mathfrak{U}_r^-$ preserves fibers; accordingly, the action of $\on{Lie}(\mathfrak{U}_r^-)$ is by vertical vector fields. 
\begin{corollary}
The vector fields $\varrho(X)$ 	for $X=\sum_{i=1}^r X_i\epsilon^i \in \on{Lie}(\mathfrak{U}_r^-) $ are complete. 	
\end{corollary}
\begin{proof}
	Since $X^{r+1}=0$ as an operator on $C^\infty(M)\otimes \AA_r$, the 1-parameter group 
	$t\mapsto U(t)=\exp(tX)\in \mf{U}_r^-$ is well-defined. Its action on $T_rM$ is a 1-parameter group of diffeomorphisms of $T_rM$, giving the flow of $X$. 
\end{proof}
Since the group $\mathfrak{U}_r^-$ preserves fibers of $T_rM$, it acts on the space $\Gamma(T_rM)$ of sections. This space has a base point given by the `zero section'   $\on{ev}\colon M\to T_rM,\ m\mapsto \on{ev}_m$. 
 \begin{lemma}
 The action of $\mathfrak{U}_r^-$ on the space of sections of $T_rM$ is free and transitive. Its application to the zero section  hence gives a bijection
 \[ \mathfrak{U}_r^-\to \Gamma(T_rM).\]
 Similarly, the map $X\mapsto \varrho(X)|_M\mod TM$ gives an isomorphism 
 \[ \on{Lie}(\mathfrak{U}_r^-)\to \Gamma((T_rM)_\lin).\]
 \end{lemma}
 \begin{proof}	
 Recall from Lemma \ref{lem:ker} that the kernel of the map $\mathfrak{U}_r^-\to \mathfrak{U}_{r-1}^-$	
 consists of elements $\on{id}+X\epsilon^r$ where $X$ is a vector field.  
 Its action on the fibers of $\Gamma(T_rM)\to \Gamma(T_{r-1}M)$ is free 
 and transitive. By induction, this implies that the action of $\mathfrak{U}_r^-$ on the fibers of $\Gamma(T_rM)\to \Gamma(T_0M)=\{0_M\}$ is free and transitive. 
 
 For the second part, recall $(T_rM)_\lin=\nu(T_rM,M)$. The map $\on{Lie}(\mathfrak{U}_r)\to \Gamma(T(T_rM)|_M),\ X\mapsto \varrho(X)$ is a bijection, as is immediate from the coordinate description (and also from the result for $\mathfrak{U}_r$). It restricts to a bijection $\mf{X}(M)\otimes \R\subset \on{Lie}(\mathfrak{U}_r)$ to $\Gamma(TM)=\Gamma(T(T_0M)|_M)$, and hence 
 descends to a bijection $\on{Lie}(\mathfrak{U}_r^-)\to \Gamma((T_rM)_\lin)$
 \end{proof}

\begin{remark}
The proof gives a bijection 
\[ \on{Lie}(\mathfrak{U}_r)=\Gamma(TM\otimes \AA_r)\to \Gamma(T(T_rM)|_M),\ X\mapsto \varrho(X)|_M.\]
%\sum_{i=0}^r X_i \epsilon^i\mapsto  	\sum_{i=0}^r X_i^{(-i)}|_M.\] 
%
One readily checks that this map is $C^\infty(M)$-bilinear, and hence gives isomorphisms of vector bundles 
$TM\otimes \AA_r\to T(T_rM)|_M$ and $TM\otimes \AA_r^-\to (T_rM)_\lin$.
\end{remark}

%For $f\in C^\infty(M)$,  \[ \sum_{i=0}^r (f\,X_i)^{(-i)}|_M=\sum_{i=0}^r\sum_{j=0}^{r-i} f^{(j)}|_M X_i^{(-i-j)}|_M=f\, \sum_{i=0}^r X_i^{(-i)}|_M\]since $f^{(0)}|_M=f$, while $f^{(j)}|_M=0$ for  $j>0$. 

%\begin{lemma}	For a graded subbundle $P\subset T_rM$, the space of sections $\Gamma(P)$ is a subgroup of $\Gamma(T_rM)$ if and only if $\Gamma(P_\lin)$ is a Lie subalgebra of 	$\Gamma((T_rM)_\lin)$. \end{lemma}\begin{proof}	Suppose $P\subset T_rM$ is a graded subbundle such that $\Gamma(P)$ is a subgroup. By taking the logarithm  	of sections $U\in \Gamma(P)$ we obtain the sub-Lie algebra of  elements of 	$\mf{X}(M)\otimes \AA_r^-$ whose image under $\varrho$ is  tangent to $P$. This space is isomorphic to sections of the normal bundle $P_\lin=\nu(P,M)$ by the restriction of 	vector fields  to $M\subset P\subset T_rM$.  \end{proof}
Finally, let us note the following fact. 
\begin{proposition}
The tangent action of $\mathfrak{U}_r$ on $T(T_rM)$ commutes with the Koszul action of the algebra $\AA_r$.  	
\end{proposition}
\begin{proof}
	This follows from the description 
\[ T(T_rM)=\Hom_{\AA_r-\alg}(C^\infty(M)\otimes \AA_r,\AA_1\otimes \AA_r)\]	
since the $\mathfrak{U}_r$-action is defined by $\AA_r-\alg$ automorphisms of $C^\infty(M)\otimes \AA_r$ while the Koszul action is defined by $\AA_r-\alg$ homomorphisms of $\AA_1\otimes \AA_r$. 
\end{proof}

\subsection{Weightings in terms of $T_rM$}
The description of weightings in terms of the $r$-th tangent bundle is as follows. 

\begin{theorem}\cite{loi:wei} Given an order $r$ weighting of $M$ along $N$, there is a unique graded subbundle $Q\subset T_rM$ along $N\subset M$, with the property that for all $0<i\le r$, 
\begin{equation}\label{eq:filtration_fromq}	C^\infty(M)_{(i)}=\{f\colon \ f^{(i-1)}|_Q=0\}.\end{equation}
(The ideals for $i>r$ are determined by \eqref{eq:extra}.) 	
The non-positive part of the filtration on vector fields is described in terms of $Q$ as 
\begin{equation}\label{eq:vf_filtration}
 \mf{X}(M)_{(-j)}=\{X\colon X^{(-j)}\mbox{  is tangent to $Q$ }\},\end{equation}
for $j=0,\ldots,r$. The weighted normal bundle $\nu_\W(M,N)$  is the quotient of 
$Q$ under the equivalence relation 
\[ q_1\sim q_2 \Leftrightarrow
\forall   f\in C^\infty(M)_{(i)}\colon f^{(i)}(q_1)=f^{(i)}(q_2).\]
\end{theorem}
In local weighted coordinates
$x_a\in C^\infty(U)$, the submanifold $Q$ is described by the vanishing of all coordinates $x_a^{(i)}\in C^\infty(T_rU)$ such that $w_a>i$.  The quotient map forgets the coordinates for which $w_a<i$, while $x_a^{(w_a)}$ descend to the coordinates 
$x_a^{[w_a]}$ on the weighted normal bundle. The vertical bundle of the fibration $Q \rightarrow \nu_\W(M,N)$ is $\epsilon \cdot (TQ)\subset TQ$, where the dot indicates the Koszul action.

\begin{remark}
Let $Q_i\subset T_iM$ defined as pre-images of $Q$ under $T_iM\to T_rM$ 
if $i>r$  and as images under $T_rM\to T_iM$ if $i<r$. This gives a tower of graded bundles 
\[ \cdots Q_{r+1}\to Q_r\to Q_{r-1}\to \cdots Q_0=N,\] 
and $C^\infty(M)_{(i)}$ may be described for \emph{all} $i$ as the functions for which $f^{(i-1)}$ vanishes on $Q_{i-1}$.
\end{remark}

%\begin{remark} As for any graded bundle, we may ask about its linear approximation. For the $r$-th tangent bundle one finds $ (T_rM)_\lin=TM\oplus \cdots \oplus TM$, where the $(T_rM)_\lin^{-i}$ for $i=1,\ldots,r$ is spanned by restrictions $X^{(-i)}|_M$, for all $X\in \mf{X}(M)$. The linear approximation of the graded subbundle  $Q$ is  \[ Q_\lin=\ti{F}_{-1}\oplus \cdots \oplus \ti{F}_{-r}\]    where $\ti{F}_{-i}\subset TM|_N$ is the pre-image of $F_{-i}\subset \nu(M,N)$ under the quotient map. \end{remark}

%
Not every graded subbundle $Q\subset T_rM$ arises from a weighting. One necessary condition is that the tangent bundle 
$TQ$ must be invariant under the Koszul $\AA_r$-action (whenever  
$X^{(-j)}$ is tangent to $Q$ then so is $X^{(-j-1)}$). Further, by our conventions $Q$ must be the pre-image of a subbundle $Q'\subset T_{r-1}M$, i.e., it must be $TM$-invariant.

\begin{theorem}\cite{loi:wei} 
Suppose $Q\subset T_rM$ is a graded subbundle, invariant under the action of $TM$ and such that $TQ$ is invariant under the Koszul action. Then $Q$ comes from an order $r$ weighting if and only if the subgroup $(\mathfrak{U}_r)_Q$ preserving $Q$ acts locally transitively on $Q$. 
\end{theorem}

Another way of putting the last condition (and indeed the way it was formulated in \cite{loi:wei}) is that $TQ$ is spanned by the collection of all lifts $X^{(-i)}$, for $i=0,\ldots,r$ and $X\in\mf{X}(M)$, with the property that $X^{(-i)}$ is tangent to $Q$.

%\begin{remark}The bundle $Q$ is fully determined by its image $Q'\subset T_{r-1}M$. Indeed, the filtration on functions may also be described by \eqref{eq:filtration_fromq} with $Q'$ replacing $Q$. One reason for working with $Q$ is that it gives $\nu_\W(M,N)$ as a quotient of $Q$. 	\end{remark}

\subsection{Singular Lie filtrations as singular foliations of $T_rM$}
We shall now turn to the interpretation of singular Lie filtrations in terms of singular foliations of higher tangent bundles. 
Observe that a singular Lie filtration $\mf{X}_M=\H_{-r}\supset\cdots \supset \H_{0}$ 
determines a graded Lie subalgebra  
\[ \on{Lie}(\mathfrak{U}_{r,\H})\subset \on{Lie}(\mathfrak{U}_r)\subset \mf{X}(M)\otimes \AA_r^,\] 
consisting of all $X=\sum_{j=0}^r X_j \epsilon^j$ with $X_j\in \H_{-j}$. 
It follows that 
\[ \D_\H(M)=C^\infty(T_rM)\cdot \{ \varrho(X)|\ X\in \on{Lie}(\mathfrak{U}_{r,\H})\}
\subset \mf{X}(T_rM)\] 
is locally finitely generated and involutive: $[\D_\H,\D_\H]\subseteq\D_\H$. 
%That is, $\D_\H$ defines a singular foliation of $T_rM$. Put differently, this foliation is defined by the orbits of the subgroup 
%\[ \mf{U}_{r,\H}\subset \mathfrak{U}_r\] 
%obtained by exponentiating $\on{Lie}(\mf{U}_{r,\H})$. 
This singular foliation is $\AA_r$-invariant (since $\H_{-j}\subset\H_{-j-1}$),  $TM$-invariant (since $\H_{-r}(M)=\mf{X}(M)$), and $(\R,\cdot)$-invariant (since the lifts $X^{(-j)}$ are homogeneous).  
If $\H_\bullet$ is a regular Lie filtration, so that $\H_{-j}=\Gamma(H_{-j})$, then
$\D_\H$ is a regular foliation of rank equal to $\sum_j \on{rank}(H_{-j})$.  For the following result, we assume $\H_0=0$.

\begin{theorem}
Let $M$ be a manifold with a singular Lie filtration $\H_{-r}\cdots \supset \cdots \supset \H_{-1}\supset 0$, and let $N\subset M$ be an $\H_\bullet$-clean submanifold. Then the graded subbundle $Q\subset T_rM$ corresponding to the weighting along $N$ is given by 
\begin{equation}\label{eq:q} Q=\mathfrak{U}_{r,\H}^-\cdot T_rN.\end{equation}
Its linear approximation is 
\[ Q_\lin=\ti{F}_{-1}\oplus\cdots \oplus \ti{F}_{-r}\]
where $\ti{F}_{-i}|_m=T_mN+\H_{-i}|_m$. 
\end{theorem}
\begin{proof}
Let $Q\subset T_rM$ be the graded subbundle defined by the weighting. We have $Q=(\mf{U}_r^-)_Q\cdot T_rN$ since this is true for any weighting.  Since $\mathfrak{U}_{r,\H}^-\subset (\mf{U}_r^-)_Q$, this proves the inclusion $\supset$ in \eqref{eq:q}. For the opposite inclusion, we use a dimension count. Recall that $\dim Q=\sum_{i=0}^r k_i$ where $k_i=\dim F_{-i}$. On the other hand, given $m\in N$, choose  
an open neighborhood $U\subset M$ and a local frame $V_1,\ldots,V_n\in \mf{X}(U)$ of $TM|_U$, 
with the property that $V_1,\ldots,V_{k_0}$ are tangent to $N$, and $V_a\in \H_{-i}(U)$ for $k_0<a\le k_i$. 
Then $V_a^{(i)}$ for $1\le a\le n$ and $0\le i\le r$ are a local frame 
for  $T(T_rU)$. Since the $V_a^{(i)}$ for $a\le k_0$ restrict to a frame for $T(T_r(N\cap U))$, it follows that 
the  $V_a^{(i)}$ for $a>k_0$ span a complement of $T(T_r(N\cap U))$. We hence see that at any point  $x\in T_rN$, 
with base point $m\in N$, the tangent space to $T_rN$ together with the orbit directions for the  $\mathfrak{U}_{r,\H}^-$-action span subspaces 
of dimension 
\begin{align*}
\lefteqn{\dim T_rN+\sum_{i=1}^r \big(\dim \H_{-i}|_m-\dim (\H_{-i}|_m\cap T_mN)\big)}\\
&=
(r+1)\dim N+\sum_{i=1}^r \big(\on{dim}\wt{F}_{-i}|_m-\dim T_mN \big)\\
&=(r+1)k_0+\sum_{i=1}^r  (k_i-k_0)=k_0+\ldots+k_r=\dim Q \qedhere
\end{align*}
\end{proof}

\section{Weighted normal bundles from singular Lie filtrations}
Given a regular Lie filtration $TM=H_{-r}\supset \cdots \supset H_{-1}\supset 0$, the associated graded 
bundle $\mf{p}=\on{gr}(TM)$ inherits a fiberwise Lie bracket, turning it into a family of nilpotent Lie algebras \cite{tan:dif}. In \cite{erp:tan}, this is called the \emph{osculating Lie algebroid}, and the family of 
Lie groups $P\to M$ integrating it is called the \emph{osculating Lie groupoid}. Given an $H$-filtered submanifold 
$N\subset M$, there is the osculating Lie groupoid $R\to N$ for the induced Lie filtration on $N$.  
Haj-Higson proved that the weighted normal bundle $\nu_\W(M,N)$ is the quotient
$P|_N/R$. We will generalize this observation to singular Lie filtrations. 

\subsection{Lie algebras from singular Lie filtrations}
Suppose $M$ is a manifold equipped with a singular Lie filtration $\mf{X}_M=\H_{-r}\supset\cdots \supset \H_{-1}\supset \H_0\supset 0$. 
Consider the  sheaf of negatively graded Lie algebras
\begin{equation} 
\label{eq:assocgr}
\bigoplus_{i=1}^r \H_{-i}/\H_{-i+1}.
\end{equation}
Pulling \eqref{eq:assocgr} back to a given point $m\in M$ (as $C^\infty_M$-modules), we obtain a  negatively graded vector space
\begin{equation} 
\label{eq:assocgrfibres}
\mf{p}_m=\bigoplus_{i=1}^r \mf{p}^{-i}_m, \qquad \mf{p}^{-i}_m=\H_{-i}/(\H_{-i+1}+\I_m\H_{-i}).
\end{equation}
where  $\ca{I}_m\subset C^\infty_M$ is the vanishing ideal. 
\begin{lemma}\label{lem:descent}
The Lie bracket on vector fields descends to Lie brackets on the vector spaces $\mathfrak{p}_m$, compatible with the grading. 	
\end{lemma}
\begin{proof}
For $0< i,j\le r$, we have 
\[ [\H_{-i+1}+\I_m\H_{-i},\H_{-j}]\subset \H_{-i-j+1}+\I_m\H_{-i-j}+\H_{-i}\subset  \H_{-i-j+1}+\I_m\H_{-i-j}.\]
We conclude that the bracket $[\cdot,\cdot]\colon \H_{-i}\times \H_{-j}\to \H_{-i-j}$ descends to the quotients, making \eqref{eq:assocgrfibres} into a negatively graded Lie algebra. 
\end{proof}

%The map\begin{equation}\label{eq:quotientmap}	\on{Lie}(\mf{U}_{r,\H}^-)=\{\sum_i X_i\epsilon^i|\ X_i\in \H_{-i}\}\to\mf{p}_m,\end{equation}given by the collection of quotient maps $\H_{-i}\to \mf{p}_m^{-i}$, is a morphism of graded Lie algebras. 

\begin{example}\label{ex:regular}
Consider the case of a  \emph{regular} Lie filtration, given by a filtration of the tangent bundle $TM=H_{-r}\supset\cdots \supset H_{-1}\supset H_0$. Here the spaces $\mf{p}_m$ are the fibers of the associated graded bundle 
$\mf{p}=\on{gr}(TM)=\bigoplus H^{-i}/H^{-i+1}$, and the bracket defines a Lie algebroid structure on  $\mf{p}$, with zero anchor. Following \cite{erp:tan}, we call  $\mf{p}$  the \emph{osculating Lie algebroid} of  the filtered manifold $(M,H_{-\bullet})$. The nilpotent Lie groups $P_m$ integrating $\mf{p}_m$ define 
the \emph{osculating groupoid} $P=\bigcup_{m\in M} P_m$. 
\end{example}

\subsection{Clean submanifolds}
For a submanifold $N\subset M$, let $\H_{-i}^N\subset \H_{-i}$ be the subsheaf of vector fields in $\H_{-i}$ that are furthermore tangent to $N$. Since $[\H_{-i}^N,\H_{-j}^N]\subset \H_{-i-j}^N$, this defines a singular Lie filtration $\H_\bullet^N$. Replacing $\H$ with $\H^N$ in the definition of $\mf{p}_m$, we obtain a graded Lie subalgebra
\[\mf{r}_m\subset \mf{p}_m\]
where $\mf{r}_m^{-i}\subset \mf{p}_m^{-i}$ is the image of $\H_{-i}^N$ under the quotient map.  
Let $R_m\subset P_m$ be the nilpotent Lie groups integrating $\mf{r}_m\subset \mf{p}_m$.

\begin{example}\label{ex:weightingexample}
	Given a weighting of $M$ along a submanifold $N$, let $\K_{-i}=\mf{X}_{M,(-i)}$ for $i=0,\ldots,r$. 
	(See Example \ref{ex:weightinggivesliefiltration}.) 
	Then the graded Lie algebra bundle $\k=\gr(\mf{X}_M)^-\to N$ 
	from Section \ref{subsec:homogeneous} may be described, for all $m\in N$, as  
\begin{equation}\label{eqref:k2} \k_m=\bigoplus_{i=1}^r \k_m^{-i},\ \ \k_m^{-i}=\K_{-i}/(\K_{-i+1}+\I_m\K_{-i}),\end{equation}
	which is a special case of the construction of $\mf{p}_m$. 
	The images of $\K_{-i}^N$ define the summands 	of the graded Lie subalgebras 
	$\mf{l}_m\subset \k_m$. 
\end{example}

\begin{theorem}\label{th:hh}
Let $M$ be a manifold with a singular Lie filtration $\H_\bullet$, and $N\subset M$ an $\H_\bullet$-clean submanifold, 
with the corresponding weighting of $M$ along $N$. Then the fibers of the weighted normal bundle are 
\[ \nu_\W(M,N)|_m=P_m/R_m.\]
\end{theorem}

\begin{proof}
The singular Lie filtration $\H_\bullet$ determines a weighting of $M$ along the $\H_\bullet$-clean submanifold $N$. We shall use the notation from Example \ref{ex:weightingexample} for this weighting. The inclusion maps 
\begin{equation}\label{eq:inclusionmaps} 
\H_{-i}\hra \K_{-i},\ \ \ i=0,\ldots,r
\end{equation}
 determine a morphism of sheaves of graded Lie algebras $\on{gr}(\H)\to \on{gr}(\K)$; upon pullback to 
$m\in N$ this becomes a Lie algebra morphism
\begin{equation}
\label{eq:pmtokm} 
\mf{p}_m\to \k_m. 
\end{equation}
The map \eqref{eq:inclusionmaps} restricts to inclusions $\H_{-i}^N\hra \K_{-i}^N,\ i=0,\ldots,r$, 
hence \eqref{eq:pmtokm} takes $\mf{r}_m$ to $\mf{l}_m$, and induces a map
\begin{equation}
\label{eq:pmtokm2}
\mf{p}_m/\mf{r}_m\rightarrow \k_m/\mf{l}_m.
\end{equation}
Since $\mf{p}_m^{-i}/\mf{r}_m^{-i}$ and $\k_m^{-i}/\mf{l}_m^{-i}$ are both identified with 
$F_{-i}/F_{-i+1}|_m$, this map is an isomorphism.

Exponentiating \eqref{eq:pmtokm} defines an action of $P_m$ on $K_m/L_m=\nu_\W(M,N)|_m$, and \eqref{eq:pmtokm2} implies that the stabilizer of this action is $R_m$. That is, 
\[ P_m/R_m\cong  K_m/L_m=\nu_\W(M,N)|_m.\]

With the additional assumption that $\dim \H_{-i}|_m$ is constant 
for $m\in N$, the Lie groups $P_m$ assemble into a smooth family, i.e., a Lie groupoid $P|_N\to N$, and similarly for  
$R_m$. 
\end{proof}

For the case of a regular Lie filtration, we saw in Example \ref{ex:regular} that the Lie algebras $\mf{p}_m$ combine into a
locally trivial vector bundle, the osculating Lie algebroid  $\mf{p}\to M$. Similarly, the induced Lie filtration 
of $N$ given by the bundles $H_{-i}|_N\cap TN$ defines the osculating Lie algebroid $\mf{r}\to N$. 
Exponentiating to the corresponding osculating Lie groupoids, we then obtain 
\begin{equation}\label{eq:newquot} \nu_\W(M,N)=P|_N/R.\end{equation}
This recovers the result of Haj-Higson \cite{hig:eul}. 

For more  singular Lie filtrations, the Lie algebras $\mf{p}_m$ do not combine into a vector bundle, unless the dimension of the graded summands $\mf{p}^{-i}_m$ are constant. In fortunate situations, this can happen along submanifolds $N\subset M$, and in this case
\[ \mf{p}|_N=\bigcup_{m\in N}\mf{p}_m\] 
will be a Lie algebroid over $N$, with integrating Lie groupoid $P|_N=\bigcup_{m\in N}P_m\to N$. If $N\subset M$ is $\H_\bullet$-clean, then it follows that the $\mf{r}_m^{-i}$ have constant dimension as well, and so define a Lie algebroid 
$\mf{r}\to N$, integrating to $R\to N$.  In these cases, we again have the presentation of the weighted normal bundle as a quotient \eqref{eq:newquot}. 
\begin{example}
	Suppose $\H_\bullet$ is a singular Lie filtration, and 
	$\F$ is a singular foliation with the property $[\F,\H_{-i}]\subset \H_{-i}$ for all $i$. Then the local flow of vector fields in $\F$ acts by automorphisms of the singular Lie filtration. Hence, if $N$ is a leaf (or an open subset of a leaf) of $\F$, then 
	it is automatic that $N$ is $\H_\bullet$-clean, and that $\mf{r}\subset \mf{p}|_N$ are well-defined Lie algebroids over $N$.  
\end{example}

\bibliographystyle{amsplain} 
%\bibliography{../../../Biblio/ref}

\def\cprime{$'$} \def\polhk#1{\setbox0=\hbox{#1}{\ooalign{\hidewidth
			\lower1.5ex\hbox{`}\hidewidth\crcr\unhbox0}}} \def\cprime{$'$}
\def\cprime{$'$} \def\cprime{$'$} \def\cprime{$'$} \def\cprime{$'$}
\def\polhk#1{\setbox0=\hbox{#1}{\ooalign{\hidewidth
			\lower1.5ex\hbox{`}\hidewidth\crcr\unhbox0}}} \def\cprime{$'$}
\def\cprime{$'$} \def\cprime{$'$} \def\cprime{$'$} \def\cprime{$'$}
\providecommand{\bysame}{\leavevmode\hbox to3em{\hrulefill}\thinspace}
\providecommand{\MR}{\relax\ifhmode\unskip\space\fi MR }
% \MRhref is called by the amsart/book/proc definition of \MR.
\providecommand{\MRhref}[2]{%
	\href{http://www.ams.org/mathscinet-getitem?mr=#1}{#2}
}
\providecommand{\href}[2]{#2}

\end{document}